\documentclass[10pt,reqno]{amsart}
\usepackage{amssymb}
\usepackage{amsmath, amssymb, amsthm}
\usepackage{mathrsfs}
\numberwithin{equation}{section}
\usepackage{color}
\usepackage{hyperref}
\newtheorem{theorem}{Theorem}[section]

\newtheorem{lemma}[theorem]{Lemma}
\newtheorem{remark}{Remark}[section]
\newtheorem{corollary}[theorem]{Corollary}
\newtheorem{prop}{Proposition}[section]
\newtheorem{definition}{Definition} [section]
\newcommand{\dif}{\mathrm{d}}
\newcommand{\ddiv}{\mathrm{div}}

\begin{document}
\title[The Newtonian limit for compressible UCM model]{Global well-posedness and Newtonian limit of the compressible upper-convected Maxwell model in analytic  spaces}\author{Renyong Guan and Yuxi Hu}
 \thanks{\noindent  Renyong Guan,   Department of Mathematics, China University of Mining and Technology, Beijing, 100083, P.R. China, renyguan@163.com\\
\indent  Yuxi Hu, Department of Mathematics, China University of Mining and Technology, Beijing, 100083, P.R. China, yxhu86@163.com\\
 }
\begin{abstract}
In this paper, we investigate the compressible upper-convected Maxwell (UCM) system in \(\mathbb{R}^3\), which corresponds to the purely elastic case of the compressible Oldroyd-B model with vanishing retardation time. In this regime, the velocity equation loses the viscous dissipation term, and the constitutive nonlinearities induce a derivative loss, so that the standard approaches in Sobolev or finitely regular Besov spaces fail. We therefore work in analytic function spaces, realized as exponentially weighted hybrid Besov spaces, and develop a global-in-time analytic framework based on Littlewood--Paley analysis and uniform-in-\(\tau\) energy estimates. A key ingredient is the introduction of the kinematic stress tensor, which allows us to establish the required  a priori estimates in hybrid Besov spaces. Based on these estimates, we prove the existence of a unique global-in-time analytic solution for small analytic initial data. Moreover, for ill-prepared data, we rigorously justify the global-in-time relaxation (Newtonian) limit as the relaxation parameter \(\tau\to0\), with an explicit convergence rate of order \(O(\sqrt{\tau})\).
 \\[0.5em]

\end{abstract}
\maketitle
\section{Introduction}
\subsection{The model}
\label{sec:model}

Viscoelastic fluids are complex fluids that exhibit both viscous and elastic responses. Typical examples range from incompressible viscoelastic fluids, such as polymer solutions and melts, blood and liquid-crystalline polymers, to compressible viscoelastic fluids, such as polymer foams, compressible polymer melts in foaming processes, shock-loaded polymers, and geophysical flows.The essential feature of such fluids is that the stress at a material point depends not only on the instantaneous strain rate but also on the deformation history; this memory effect is characterized by relaxation and retardation times.

The Oldroyd-B (OB) model, introduced by Oldroyd \cite{ref4}, is one of the most fundamental constitutive models for viscoelastic fluids.  It is formulated as follows:
\begin{equation}\label{odmt}
	\tau \frac{\mathcal{D}\mathbf{T}}{\mathcal{D}t}+\mathbf{T}
	=2\mu_0\left(\mathbf{D}(\mathbf{u})+\tilde{\tau}\frac{\mathcal{D}\mathbf{D}(\mathbf{u})}{\mathcal{D}t}\right),
\end{equation}
where $\mathbf{T}$ is the  symmetric stress tensor and $\mathbf{u}$ is the fluid velocity. $\tau$, $\mu_0$, and $\tilde{\tau}$ denote the relaxation time, kinematic viscosity, and retardation time, respectively; and $\frac{\mathcal{D}}{\mathcal{D}t}$ is the Oldroyd  derivative, defined by
\[
\frac{\mathcal{D}\mathbf{T}}{\mathcal{D}t}
:=\partial_t\mathbf{T}+\mathbf{u}\cdot\nabla\mathbf{T}
+\mathbf{T} \cdot \mathbf{W}(\mathbf{u})-\mathbf{W}(\mathbf{u}) \cdot \mathbf{T}
-b(\mathbf{T}\cdot \mathbf{D}(\mathbf{u})+\mathbf{D}(\mathbf{u})\cdot\mathbf{T}),
\]
with
\[
\mathbf{D}(\mathbf{u})=\frac{1}{2}\bigl(\nabla\mathbf{u}+(\nabla\mathbf{u})^T\bigr),
\qquad
\mathbf{W}(\mathbf{u})=\frac{1}{2}\bigl(\nabla\mathbf{u}-(\nabla\mathbf{u})^T\bigr).
\]
Here $b\in[-1,1]$ is a constant. In particular, for $b=-1,0,1$, the corresponding Oldroyd constitutive relations are called the lower-convected, corotational, and upper-convected Maxwell models, respectively. Note that if we set $\tau=\tilde{\tau}=0$ in \eqref{odmt}, it reduces to the classical Newtonian fluid.

For an incompressible viscoelastic fluid, the mass conservation reduces to the divergence-free condition \(\nabla\cdot\mathbf{u}=0\). By absorbing the retardation term into the momentum equation through a modified stress tensor, the incompressible Oldroyd-B system can be written as
\begin{equation}\label{incomp_OB}
	\begin{cases}
		\nabla\cdot\mathbf{u}=0,\\
		\rho\left(\mathbf{u}_t+\mathbf{u}\cdot\nabla\mathbf{u}\right)+\nabla \pi
		-\mu_0\frac{\tilde{\tau}}{\tau}\Delta\mathbf{u}=\nabla\cdot\mathbf{T},\\
		\tau \frac{\mathcal{D}\mathbf{T}}{\mathcal{D}t}+\mathbf{T}
		=2\mu_0\left(1-\frac{\tilde{\tau}}{\tau}\right)\mathbf{D}(\mathbf{u}),
	\end{cases}
\end{equation}
Here \(\pi\) denotes the pressure, which acts as the Lagrange multiplier associated with the incompressibility constraint \(\nabla\cdot\mathbf{u}=0\).

For a compressible viscoelastic fluid,  by combining the mass and momentum conservation laws and absorbing the retardation term into the momentum equation, one obtains  the compressible Oldroyd-B system
\begin{equation}\label{comp_OB}
	\begin{cases}
		\rho_t+\nabla\cdot(\rho\mathbf{u})=0,\\
		\rho\mathbf{u}_t+\rho\mathbf{u}\cdot\nabla \mathbf{u}+\nabla p
		-\mu_0\frac{\tilde{\tau}}{\tau}\bigl(\Delta\mathbf{u}+\nabla(\nabla\cdot\mathbf{u})\bigr)=\nabla\cdot\mathbf{T},\\
		\tau \frac{\mathcal{D}\mathbf{T}}{\mathcal{D}t}+\mathbf{T}
		=2\mu_0\left(1-\frac{\tilde{\tau}}{\tau}\right)\mathbf{D}(\mathbf{u}).
	\end{cases}
\end{equation}
Here and in what follows, \(p=p(\rho)\) denotes the pressure, which is given by the \(\gamma\)-law
\[
p(\rho)=A\rho^\gamma,\qquad A>0,\quad \gamma\ge 1,
\]
where \(\gamma\) is the adiabatic exponent. The additional term \(-\mu_0\frac{\tilde{\tau}}{\tau}(\Delta\mathbf{u}+\nabla(\nabla\cdot\mathbf{u}))\) in the momentum equation  is the second-order dissipative term induced by the retardation time; in the incompressible case it reduces to \(-\mu_0\frac{\tilde{\tau}}{\tau}\Delta\mathbf{u}\).  It vanishes in the purely elastic case \(\tilde{\tau}=0\).

Note that the constitutive relation \eqref{odmt} was originally derived for incompressible fluids; in the compressible case, additional terms accounting for compressibility effects should be incorporated. In particular, the volumetric stress contribution due to compressibility (i.e., \(\nabla\cdot\mathbf{u}\neq0\)) is missing in \eqref{odmt}. Motivated by the above considerations,  Bollada and Phillips \cite{ref3} (see also \cite{ref5,ref6}) considered, for compressible viscoelastic fluids,  the following rheological equation of state:
\begin{equation}\label{1.2}
	\tau\left(\overset{\triangledown}{\mathbf{T}}+(\nabla\cdot\mathbf{u})\mathbf{T}\right)+\mathbf{T}
	=\tilde{\mu}(\rho)2\mathbf{D}(\mathbf{u})+\lambda(\rho)(\nabla\cdot\mathbf{u})\mathrm{I},
\end{equation}
where  \(\overset{\triangledown}{\mathbf{T}}\) denotes the upper-convective derivative (see Oldroyd \cite{ref4}), defined by
\[
\overset{\triangledown}{\mathbf{T}}
:=\partial_t\mathbf{T}+\mathbf{u}\cdot\nabla\mathbf{T}
-\nabla\mathbf{u}\cdot\mathbf{T}
-\mathbf{T}\cdot(\nabla\mathbf{u})^T.
\]
The additional term \((\nabla\cdot\mathbf{u})\mathbf{T}\) on the left hand side of \eqref{1.2} stems from compressibility effects and is originally due to Truesdell \cite{ref7}. 
In \eqref{1.2}, the second viscosity coefficient  $\lambda(\rho)$  is assumed to satisfy the relation
\begin{align}\label{BD}
\lambda(\rho)=\rho\frac{\mathrm{d}\tilde{\mu}}{\mathrm{d}\rho}-\tilde{\mu}(\rho),
\end{align}
where  \(\tilde{\mu}(\rho)\)  denotes the shear viscosity coefficient. It is worth noting that the relation \eqref{BD} exhibits a structure analogous to that of the well-known \textit{BD entropy}, originally introduced by Bresch and Desjardins \cite{Bresch2003cmp} in the context of the compressible Navier-Stokes equations.

Recall that the constitutive equation \eqref{1.2} corresponds to the purely elastic case with vanishing retardation time, i.e., \(\tilde{\tau}=0\). Consequently, the velocity equation contains no viscous dissipation. For the subsequent uniform-in-\(\tau\) analysis, we assume that the shear viscosity is proportional to the density, i.e., \(\tilde{\mu}=\mu\rho\), where \(\mu>0\) is a constant. Equivalently, the kinematic viscosity
\[
\nu=\frac{\tilde{\mu}}{\rho}=\mu
\]
is constant, so that density variations do not affect momentum diffusion. This is consistent with the phenomenological observation that compressible polymeric fluids become ``thicker'' as the density increases. From a microscopic viewpoint, a higher density corresponds to a larger number of polymer chains or frictional units per unit volume, leading to stronger chain--chain friction and hence a larger dynamic viscosity. As a consequence of this assumption, the BD entropy-related coefficient vanishes,
\[
\lambda(\rho)=\rho\frac{\mathrm{d}\tilde{\mu}}{\mathrm{d}\rho}-\tilde{\mu}(\rho)=0,
\]
and the compressible Oldroyd-B system reduces to the target compressible UCM model:
\begin{equation}\label{yucm}
	\begin{cases}
		\rho_t +\nabla \cdot(\rho\mathbf{u})=0,\\
		\rho \mathbf{u}_t+\rho\mathbf{u}\cdot\nabla \mathbf{u}+\nabla p=\nabla\cdot\mathbf{T},\\
		\tau\left(\mathbf{T}_t+\mathbf{u}\cdot\nabla \mathbf{T}-\nabla\mathbf{u}\cdot\mathbf{T}-\mathbf{T}\cdot(\nabla\mathbf{u})^T+(\nabla\cdot\mathbf{u})\mathbf{T}\right)+\mathbf{T}=2\mu\rho\mathbf{D}(\mathbf{u}).
	\end{cases}
\end{equation}
In this paper, we study the Cauchy problem for system \eqref{yucm} in \((t,x)\in (0,\infty)\times\mathbb{R}^3\). We aim to establish its uniform-in-\(\tau\) global well-posedness and rigorously justify the relaxation limit in analytic function spaces. We shall review the relevant literature in Section 1.2, outline the main strategy in Section 1.3, and state the main results in Section 1.4.

\subsection{Related works}
For the Oldroyd-B model with positive retardation time (\(\tilde{\tau}>0\)), the momentum equation contains a second-order dissipative term of the form \(-\mu\frac{\tilde{\tau}}{\tau}\Delta\mathbf{u}\) (together with its compressible counterpart), which provides enough regularity to compensate for the derivative loss caused by the constitutive nonlinearities. Consequently, the well-posedness theory in this regime can be developed within Sobolev or Besov spaces of finite regularity.

For the incompressible Oldroyd-B model, this theory is now relatively mature. Guillop\'e and Saut \cite{refp13} pioneered the proof of local well-posedness and the global existence of small smooth solutions, while Lions and Masmoudi \cite{refp20} constructed global large renormalized weak solutions in the corotational case, and Chemin and Masmoudi \cite{ref2} established global well-posedness in critical \(L^p\) frameworks. For further extensions to exterior domains and critical Besov spaces, see \cite{molinet2004global,hieber2012global,fang2013global,refp11,refp8} and the references therein.

For the compressible Oldroyd-B model with \(\tilde{\tau}>0\), the theory is less complete but has developed rapidly. Fang and Zi \cite{fang2013strong} proved the local existence of strong solutions with large data and vacuum. Zi \cite{zi2017global} and Guo et al. \cite{guo2024stability} obtained global well-posedness and stability in critical spaces and periodic domains for small data. Wang and Wen \cite{wang2020cauchy} established global strong solutions near a non-vacuum equilibrium in \(H^s(\mathbb{R}^3)\). Recently, Liu et al. \cite{liu2023large} constructed global large solutions for viscous Oldroyd-B model; see also \cite{refp14,refp25,refp27} for related well-posedness results.


Concerning singular limits for Oldroyd-B type models, while various limits (such as the low Mach number limit and its combinations) have been studied for the classical Oldroyd-B model \cite{refp17,guillope2010regular,ren2015strong,fang2014incompressible,zhang2025combined}, the rigorous justification of the relaxation limit (\(\tau \to 0\)) for compressible UCM model \eqref{yucm} in analytic function spaces remains an open problem. This is the main goal of the present paper.

A particularly challenging branch is the Oldroyd-B model without retardation (\(\tilde{\tau}=0\)), where the velocity equation loses the second-order dissipative term entirely, making the standard Sobolev-space approach fail due to the derivative loss. For the incompressible case, Zi \cite{refp30} developed the analytic framework for the 3D incompressible Oldroyd-B model and studied its vanishing viscosity limit. For the compressible counterpart, Pan \cite{pan} extended this analytic framework  and established global analytic solutions near a non-vacuum equilibrium by exploiting the pure damping effects of the linearized equations. Building on this, Pan and Yang \cite{pan2024globallyexponential} further studied the case on the torus \(\mathbb{T}^3\), deriving exponential decay and rigorously proving the inviscid limit.

Recall that in Section 1.1, the compressible UCM model was derived from first principles by Bollada and Phillips \cite{ref3}, whose constitutive equation is given by \eqref{1.2}. Under the assumption of constant kinematic viscosity, this model reduces to system \eqref{yucm}, which is the focus of the present paper. Compared with the extensive studies on the incompressible Oldroyd-B model or the classical compressible Navier--Stokes equations, the mathematical theory for this specific model remains relatively less explored. 

A first well-posedness result was obtained by Boyaval \cite{ref9}, who reformulated the compressible UCM system as a symmetric-hyperbolic system of conservation laws by introducing an additional material metric variable, and established the local well-posedness of classical solutions in Sobolev spaces via the Godunov--Mock theorem. Subsequently, Wang, Boyaval, and Hu \cite{refaml} proved that solutions blow up in finite time for large initial data by employing a convex entropy constructed in \cite{ref9}, and obtained the global well-posedness in the one-dimensional case in \cite{ref10}. Although Pan \cite{pan} established global analytic solutions for the compressible UCM model without retardation, the relaxation parameter \(\tau\) was normalized to a fixed constant in that work. Consequently, the uniform-in-\(\tau\) estimates and the rigorous justification of the relaxation limit (i.e., the Newtonian limit) as \(\tau \to 0\) for system \eqref{yucm} have not been addressed yet. This motivates the present paper.

\subsection{Strategy}

The compressible UCM system \eqref{yucm} leads to a loss of derivative when performing energy estimates, since the momentum equation has no viscous dissipation in the purely elastic case \(\tilde{\tau}=0\).This derivative loss prevents the standard Sobolev and finitely regular Besov space approaches from working. Although Boyaval \cite{ref9} reformulated the system as a symmetric-hyperbolic one by introducing an additional material metric variable, the resulting system is structurally complicated and does not readily yield \(\tau\)-independent estimates; consequently, the global well-posedness and the relaxation limit remain unresolved in that framework. We therefore work directly in analytic function spaces. Note that the global analytic framework has already been developed for the incompressible Navier-Stokes equations \cite{refz5,refz36,refz37},  see \cite{refz43} for the Prandtl equations  and  \cite{refp30,pan} for Oldroyd-B type models. Our proof strategy consists of three key ingredients.

First, motivated by Bollada and Phillips \cite{ref3}, we reformulate the system by introducing the kinematic stress tensor \(\mathbf{S}=\mathbf{T}/\rho\). This change of variables is not merely a technical convenience; it is essential for obtaining uniform-in-\(\tau\) estimates. Indeed, the natural scaling of the stress equation requires the energy functional to be equipped with a \(\sqrt{\tau}\) weight on the stress. If the momentum equation is linearized with \(\mathbf{T}\) as the unknown, the term \(\ddiv\mathbf{T}\) generates an additional nonlinear contribution of the form \((1/\rho-1)\ddiv\mathbf{T}\). Within the analytic Besov framework, this term cannot be bounded uniformly in \(\tau\), since the \(\sqrt{\tau}\) weight associated with \(\mathbf{T}\) is incompatible with the factor \(1/\rho-1\). The transformation \(\mathbf{S}=\mathbf{T}/\rho\) removes this obstruction. Indeed, using
\[
\ddiv\mathbf{T}=\rho\,\ddiv\mathbf{S}+\mathbf{S}\nabla\rho,
\]
we see that the nonlinear term in \((1/\rho)\ddiv\mathbf{T}\) is
\(\mathbf{S}\frac{\nabla\rho}{\rho}\), which is structurally analogous to the original nonlinearities and can be treated within the same analytic framework.

Second, following a procedure similar to that in \cite{pan,refp30} and paying particular attention to the \(\tau\)-independent estimates, we use the analytic energy method with the weight \(e^{\Phi(t,\Lambda)}\), where \(\Phi(t,\xi)=(2\lambda_0-\lambda\theta(t))|\xi|\). This weight not only encodes the analytic radius but also, after commuting with the equations, generates an additional higher-order damping term of the form \(\lambda\dot{\theta}(t)|\xi|\). This damping is crucial because it compensates for the one-derivative loss caused by the nonlinear terms. To adapt the analytic radius dynamically, the function \(\theta(t)\) is defined self-consistently through the dissipative estimate itself, typically as \(\dot{\theta}(t)=\|(a_\Phi,\mathbf{u}_\Phi,\mathbf{S}_\Phi)\|_{\dot{B}_{2,1}^{3/2}}\) with \(\theta(0)=0\), so that the analytic radius decreases exactly at the rate dictated by the solution's dissipation. To handle the nonlinear products, we combine this analytic weight with the Littlewood--Paley decomposition and Bony's para-product decomposition, which splits the products into paraproduct and remainder terms. By performing a high-low frequency decomposition, we treat the regimes \(k\ge k_0\) and \(k\le k_0\) separately. The high frequencies exhibit a damping effect that yields an estimate in \(\dot{B}_{2,1}^{3/2}\), whereas the low-frequency dissipation is weaker and requires an estimate in \(\dot{B}_{2,1}^{1/2}\). This frequency-dependent behavior is precisely why the hybrid Besov norm \(\dot{B}_{2,1}^{1/2,3/2}\) is the natural energy space for system \eqref{yucm}. Consequently, the energy functional is designed such that the dissipation estimate for \(\mathbf{S}\) is measured in \(L^1(\dot{B}_{2,1}^{3/2})\), while the stress itself appears with the necessary \(\sqrt{\tau}\) weight.

Third,  to rigorously justify the Newtonian limit,  the basic regularity of the limit equations (classical NS equations) is insufficient to support the error estimates when taking the relaxation limit. Therefore, we first need to establish higher-order regularity for the limit equations (i.e., the result in Theorem \ref{jd1.3}).  A key structural difficulty arises in this process: since the stress in the UCM model satisfies the Maxwell constitutive equation, while the classical NS equations satisfy the Newtonian constitutive equation, taking the difference directly leads to a mismatch in the stress equation. To resolve this, we introduce the term \(\tau \mathbf{S}^0_t\) into the Newtonian equation and exploit the higher regularity of the limiting solutions \((a^0, \mathbf{u}^0)\) to obtain the desired estimates. It is worth noting that these estimates are obtained in standard Besov spaces rather than in hybrid ones, whereas the dissipative estimates derived from the error system itself naturally fall into hybrid Besov spaces. This mismatch also complicates the nonlinear terms and prevents direct closure. A slightly different scaling for the \(\tilde{\mathbf{S}}\) term in the a priori estimates confines its estimate to standard Besov spaces and yields an explicit convergence rate \(O(\sqrt{\tau})\); the same framework also handles ill-prepared initial data, whose initial mismatch produces only a rapidly decaying boundary layer and does not affect the macroscopic limit as \(\tau\to0\).

\subsection{Main results}
Following the strategy outlined above, we now perform the reformulation. Under the assumption of constant kinematic viscosity, i.e., \(\mu = \tilde{\mu}/\rho\), we introduce the kinematic stress tensor \(\mathbf{S} := \mathbf{T}/\rho\). Substituting this into \eqref{1.2} and using the mass conservation, the constitutive relation reduces to
\[
\tau\overset{\triangledown}{\mathbf{S}}+\mathbf{S}=2\mu\mathbf{D}(\mathbf{u}).
\]
Note that the coefficient on the right-hand side  of the above equation is the constant \(2\mu\) rather than a density-dependent quantity. This constant coefficient is essential for establishing uniform-in-\(\tau\) a priori estimates. Therefore, we obtain the following system:
\begin{equation}\label{ysys}
	\begin{cases}
		\rho_t +\nabla \cdot(\rho\mathbf{u})=0,\\
		\rho \mathbf{u}_t+\rho\mathbf{u}\cdot\nabla \mathbf{u}+\nabla p=\rho\ddiv\mathbf{S}+\mathbf{S} \cdot \nabla \rho,\\
		\tau\left(\mathbf{S}_t+\mathbf{u}\cdot\nabla \mathbf{S}-\nabla\mathbf{u}\cdot\mathbf{S}-\mathbf{S}\cdot(\nabla\mathbf{u})^T\right)+\mathbf{S}=2\mu\mathbf{D}(\mathbf{u}).
	\end{cases}
\end{equation}
We are interested in the Cauchy problem for system \eqref{ysys} with initial data
\begin{equation}\label{idata}
	(\rho, \mathbf{u}, \mathbf{S})(0,x)=(\rho_0, \mathbf{u}_0, \mathbf{S}_0) (x),\quad x\in \mathbb R^3.
\end{equation}


Applying the analytic framework to \eqref{ysys}--\eqref{idata} with small analytic perturbations around the non-vacuum equilibrium \((1,0,0)\), we obtain the following theorem.

\begin{theorem}\label{th1.1}
	Let $\lambda_0>0$ be a fixed constant. Assume that 
	\begin{equation}\label{taura}
	0\leq	\tau\leq\min\{1, \mu\}
	\end{equation}
	and the initial data $(\rho_0, \mathbf{u}_0, \mathbf{S}_0)$ satisfy $e^{2\lambda_0\Lambda}(\rho_0-1, \mathbf{u}_0, \sqrt{\tau}\mathbf{S}_0)\in \dot{B}^{\frac{1}{2}}_{2,1}\cap\dot{B}^{\frac{3}{2}}_{2,1}$. Then there exist two positive constants $\varepsilon_0$ and $C$, independent of $\lambda_0$ and $\tau$, such that for any $0<\varepsilon\leq \varepsilon_0$, if
	\begin{align}\label{Initial}
	\|e^{2\lambda_0\Lambda}(\rho_0-1, \mathbf{u}_0, \sqrt{\tau}\mathbf{S}_0)\|_{\dot{B}^{\frac{1}{2}}_{2,1}\cap\dot{B}^{\frac{3}{2}}_{2,1}}\leq \varepsilon\lambda_0,
	\end{align}
	then the initial problem \eqref{ysys}-\eqref{idata} admits a unique solution $(\rho, \mathbf{u}, \mathbf{S})(t,x)$ satisfying, 
	\begin{align}
		&\|e^{\lambda_0\Lambda}(\rho-1, \mathbf{u}, \sqrt{\tau}\mathbf{S})\|_{L^\infty(\dot{B}^{\frac{1}{2}}_{2,1}\cap\dot{B}^{\frac{3}{2}}_{2,1})}+\|e^{\lambda_0\Lambda}(\rho-1, \mathbf{u}, \mathbf{S})\|_{L^1(\dot{B}^{\frac{3}{2}}_{2,1})} \nonumber\\
		&\le C\|e^{2\lambda_0\Lambda}(\rho_0-1, \mathbf{u}_0, \sqrt{\tau}\mathbf{S}_0)\|_{\dot{B}^{\frac{1}{2}}_{2,1}\cap\dot{B}^{\frac{3}{2}}_{2,1}}. \label{Res}
	\end{align}
\end{theorem}

\begin{remark}\label{rem:tau-independent}
	It is worth emphasizing that the constant $C$ and the a priori estimates established in Theorem \ref{th1.1} are independent of the parameter $\tau$. This uniform-in-$\tau$ property is crucial, as it allows us to pass to the limit $\tau \to 0$ and naturally derive the global well-posedness for the limiting system as a direct consequence.
\end{remark}

If we set $\tau=0$, then our system reduces to the classical Navier-Stokes equations
\begin{equation}\label{jdxt}
	\begin{cases}
		\rho_t +\nabla \cdot(\rho\mathbf{u})=0,\\
		\rho \mathbf{u}_t+\rho\mathbf{u}\cdot\nabla \mathbf{u}+\nabla p=2\mu\ddiv(\rho\mathbf{D}(\mathbf{u})),
	\end{cases}
\end{equation}
and initial data
\begin{equation}\label{jdxt1.1}
(\rho, \mathbf{u})(0,x)=(\rho_0, \mathbf{u}_0)(x).
\end{equation}

Then, we have the following corollary.

\begin{corollary}\label{jd1.1}
	Assume that the initial data $(\rho_0, \mathbf{u}_0)$ satisfy $e^{2\lambda_0\Lambda}(\rho_0-1, \mathbf{u}_0)\in \dot{B}^{\frac{1}{2}}_{2,1}\cap\dot{B}^{\frac{3}{2}}_{2,1}$. Then there exist two positive constants $\varepsilon_0$ and $C_0$, independent of $\lambda_0$, such that for any $0<\varepsilon\leq \varepsilon_0$, if
	\[
	\|e^{2\lambda_0\Lambda}(\rho_0-1, \mathbf{u}_0)\|_{\dot{B}^{\frac{1}{2}}_{2,1}\cap\dot{B}^{\frac{3}{2}}_{2,1}}\leq \varepsilon\lambda_0,
	\]
	then the initial problem \eqref{jdxt}-\eqref{jdxt1.1} admits a unique solution $(\rho, \mathbf{u})(t,x)$ satisfying, 
	\begin{align*}
		&\|e^{\lambda_0\Lambda}(\rho-1, \mathbf{u})\|_{L^\infty(\dot{B}^{\frac{1}{2}}_{2,1}\cap\dot{B}^{\frac{3}{2}}_{2,1})}+\|e^{\lambda_0\Lambda}(\rho-1, \mathbf{u})\|_{L^1(\dot{B}^{\frac{3}{2}}_{2,1})}+\|e^{\lambda_0\Lambda} \mathbf{u}\|_{L^1(\dot{B}^{\frac{5}{2}}_{2,1})}\\
		&	\le C_0	\|e^{2\lambda_0\Lambda}(\rho_0-1, \mathbf{u}_0)\|_{\dot{B}^{\frac{1}{2}}_{2,1}\cap\dot{B}^{\frac{3}{2}}_{2,1}}.
	\end{align*}
\end{corollary}

Corollary \ref{jd1.1} provides the global existence and regularity estimates for initial data belonging to
\(\dot{B}^{\frac{1}{2}}_{2,1}\cap\dot{B}^{\frac{3}{2}}_{2,1}\).
However, the study of the relaxation limit requires higher spatial regularity for the limit equations.
In particular, we enlarge the initial data space to
\(\dot{B}^{\frac{1}{2}+s}_{2,1}\cap\dot{B}^{\frac{3}{2}+s}_{2,1} (s=1, 2)\)
and prove the following theorem.
\begin{theorem}\label{jd1.3}
	Let \(s\in\{1,2\}\) and $\lambda_0$ be a fixed constant. Assume that the initial data \((\rho_0, \mathbf{u}_0)\) satisfy
	\[
	e^{2\lambda_0\Lambda}(\rho_0-1, \mathbf{u}_0)\in 
	\dot{B}^{\frac{1}{2}+s}_{2,1}\cap\dot{B}^{\frac{3}{2}+s}_{2,1},
	\qquad s=1,2.
	\]
	Then there exist two positive constants \(\varepsilon_1\) and \(C\),
	independent of \(\lambda_0\), such that
	for any \(0<\varepsilon\leq \varepsilon_1\), if
	\begin{equation}\label{smalls}
	\|e^{2\lambda_0\Lambda}(\rho_0-1, \mathbf{u}_0)\|_
	{\dot{B}^{\frac{3}{2}}_{2,1}\cap\dot{B}^{\frac{5}{2}}_{2,1}}
	\leq \varepsilon\lambda_0,
	\end{equation}
	then there exists a  unique global solution \((\rho, \mathbf{u})(t,x)\) of system \eqref{jdxt}-\eqref{jdxt1.1} satisfying, for \(s=1,2\),
	\begin{align*}
		&\|e^{\lambda_0\Lambda}(\rho-1, \mathbf{u})\|_{L^\infty(
			\dot{B}^{\frac{1}{2}+s}_{2,1}\cap\dot{B}^{\frac{3}{2}+s}_{2,1})}
		+\|e^{\lambda_0\Lambda}(\rho-1, \mathbf{u})\|_{L^1(
			\dot{B}^{\frac{3}{2}+s}_{2,1})}
		+\|e^{\lambda_0\Lambda} \mathbf{u}\|_{L^1(
			\dot{B}^{\frac{5}{2}+s}_{2,1})}\\
		&\qquad\le C
		\|e^{2\lambda_0\Lambda}(\rho_0-1, \mathbf{u}_0)\|_
		{\dot{B}^{\frac{1}{2}+s}_{2,1}\cap\dot{B}^{\frac{3}{2}+s}_{2,1}}.
	\end{align*}
\end{theorem}

Having established the uniform-in-$\tau$ estimates in Theorem \ref{th1.1} and the well-posedness of the limiting system given in Theorem \ref{jd1.3}, we are now in a position to justify the convergence. The main result is stated as follows.
 
\begin{theorem} \label{th1.2}
	Let $(\rho^\tau, \mathbf{u}^\tau, \mathbf{S}^\tau)$ be the global solution
of the system \eqref{ysys}--\eqref{idata} constructed in Theorem \ref{th1.1},
and let $(\rho^0, \mathbf{u}^0)$ be the unique global solution of the
classical Navier-Stokes equations \eqref{jdxt}--\eqref{jdxt1.1} obtained
in Theorem \ref{jd1.3} under the corresponding assumptions on the initial data with $\varepsilon_0, \varepsilon_1$ sufficiently small.
	If, in addition,
	\begin{equation}\label{scjs}
		\|e^{2\lambda_0\Lambda}(\rho^\tau_0-\rho^0_0,
		\mathbf{u}^\tau_0-\mathbf{u}^0_0,
		\sqrt{\tau}\mathbf{S}^\tau_0-2\mu\sqrt{\tau}\mathbf{D}(\mathbf{u}^0_0))
		\|_{\dot{B}^{\frac{3}{2}}_{2,1}}\leq C\sqrt{\tau},
	\end{equation}
	then, 
	\begin{align*}
		&\|e^{\lambda_0\Lambda}(\rho^\tau-\rho^0, \mathbf{u}^\tau-\mathbf{u}^0, \sqrt{\tau}\mathbf{S}^\tau-2\mu\sqrt{\tau}\mathbf{D}(\mathbf{u}^0))\|_{L^\infty(\dot{B}^{\frac{3}{2}}_{2,1})}\\	&\quad\quad+\|e^{\lambda_0\Lambda}(\rho^\tau-\rho^0, \mathbf{u}^\tau-\mathbf{u}^0)\|_{L^1(\dot{B}^{\frac{5}{2},\frac{3}{2}}_{2,1})}	+\|e^{\lambda_0\Lambda}(\mathbf{S}^\tau-2\mu\mathbf{D}(\mathbf{u}^0))\|_{L^1(\dot{B}^{\frac{3}{2}}_{2,1})}	\\
		&\leq C\|e^{2\lambda_0\Lambda}( \rho ^\tau_0-\rho^0_0,\mathbf{u}^\tau_0-{\mathbf u}^0_0,\sqrt{\tau}{\mathbf S}^\tau-2\mu\sqrt{\tau}\mathbf{D}(\mathbf{u}^0_0))\|_{\dot B_{2,1}^{\frac{3}{2}}}
+C\sqrt{\tau}\\
		&\leq C\sqrt{\tau}.
	\end{align*}
	Thus, as $\tau \to 0$,
	\[
	e^{\lambda_0\Lambda}(\rho^\tau-\rho^0, \mathbf{u}^\tau-\mathbf{u}^0)\longrightarrow 0 \quad \text{strongly in}\quad L^\infty(\dot{B}^{\frac{3}{2}}_{2,1}),
	\]
	and
	\[
e^{\lambda_0\Lambda}(\mathbf{S}^\tau-2\mu\mathbf{D}(\mathbf{u}^0))\longrightarrow 0 \quad \text{strongly in}\quad L^1(\dot{B}^{\frac{3}{2}}_{2,1}).
	\]
\end{theorem}
\begin{remark}
	It is worth emphasizing that Theorem \ref{th1.2} does not require the initial stress to satisfy the limit constraint \(\mathbf{S}^\tau_0 = 2\mu\mathbf{D}(\mathbf{u}^0_0)\) up to a higher order of \(\tau\). In fact, condition \eqref{scjs} allows an \(O(1)\) mismatch between \(\mathbf{S}^\tau_0\) and \(2\mu\mathbf{D}(\mathbf{u}^0_0)\), which precisely corresponds to the ill-prepared initial data scenario. 
\end{remark}

The remainder of this paper is organized as follows. In Section 2, we present the preliminaries, including the Littlewood--Paley decomposition, the hybrid Besov spaces, and the analytic energy framework. In Section 3, we linearize the system \eqref{ysys}, establish the a priori estimates, and prove Theorem \ref{th1.1} along with Corollary \ref{jd1.1}. In Section 4, we first establish the higher-order regularity of the limit equations in Theorem \ref{jd1.3}, which is then used to rigorously justify the relaxation limit, i.e., the convergence result stated in Theorem \ref{th1.2}.

\medskip
\noindent \textbf{Notation.} Let \(\mathcal{S}(\mathbb{R}^n)\) be the Schwartz space, and let \(\mathcal{Z}'(\mathbb{R}^n)\) be the dual of \(\mathcal{Z}(\mathbb{R}^n):=\{f\in \mathcal{S}(\mathbb{R}^n):\partial^\alpha \hat{f}(0)=0,\ \forall \alpha\in(\mathbb{N}\cup\{0\})^n\}\). For any \(f\in \mathcal{Z}'(\mathbb{R}^n)\), let \(\hat{f}\) or \(\mathcal{F}(f)\) denote its Fourier transform and \(\mathcal{F}^{-1}(f)\) its inverse. For \(s\in\mathbb{R}\), we set \(\Lambda^s:=(-\Delta)^{s/2}\). For \(a,b\in L^2\), we write \(\langle a,b\rangle\) for the \(L^2\) inner product. 

Throughout the paper, for any functional space \(X\), we use the abbreviation \(\|(f,g,\dots)\|_X := \|f\|_X + \|g\|_X + \dots\).  Unless otherwise specified, the dot product \(\mathbf{A}\cdot\mathbf{B}\) between two matrices, or between a matrix and a column vector, denotes the standard matrix multiplication.  For \(q\in[1,\infty]\) and \(T>0\), \(L^q_t(X)\) denotes the space \(L^q(0,T;X)\), and we omit the subscript \(t\) when no confusion can arise. Finally, \(C\) denotes various positive constants which may change from line to line, and we use the notation \(A\lesssim B\) as an equivalent to \(A\leq CB\). The notation \(A\approx B\) means that \(A\lesssim B\) and \(B\lesssim A\).

\section{Preliminaries: Littlewood-Paley and Bony decompositions}

Our analysis relies heavily on the Littlewood-Paley decomposition. We briefly recall the basic setup; for a comprehensive treatment, see \cite{refp1}. Let $(\chi, \varphi)$ be a pair of $C^\infty$ functions supported respectively in 
\[
\operatorname{Supp}\chi \subset \bigl\{|\xi|\le 4/3\bigr\}, \qquad
\operatorname{Supp}\varphi \subset \bigl\{3/4 \le |\xi|\le 8/3\bigr\},
\]
and satisfying 
\[
\chi(\xi)+\sum_{k\ge 0}\varphi(2^{-k}\xi)=1,\qquad
\sum_{k\in\mathbb Z}\varphi(2^{-k}\xi)=1\quad(\xi\ne 0).
\]
With $\varphi_k(\xi):=\varphi(2^{-k}\xi)$, we define the dyadic blocks and low-frequency cutoffs for $k\in\mathbb Z$:
\[
\dot{\Delta}_k u := \mathcal{F}^{-1}\bigl(\varphi(2^{-k}\cdot)\hat u(\cdot)\bigr),\qquad
\dot S_k u := \mathcal{F}^{-1}\bigl(\chi(2^{-k}\cdot)\hat u(\cdot)\bigr).
\]

The homogeneous Besov spaces are then defined as follows.

\begin{definition}
	For $s\in\mathbb R$, $(p,r)\in[1,\infty]^2$, and $u\in S'(\mathbb R^N)$, set
	\[
	\|u\|_{\dot B_{p,r}^{s}} := \Bigl\|2^{sk}\|\dot{\Delta}_k u\|_{L^p}\Bigr\|_{\ell^r}.
	\]
	The space $\dot B_{p,r}^{s}$ consists of all $u\in S'_h(\mathbb R^N)$ for which this norm is finite, where $S'_h$ denotes the subspace of temperate distributions satisfying $\lim\limits_{k\to-\infty}\|\dot S_k u\|_{L^\infty}=0$.
\end{definition}

The Littlewood-Paley decomposition offers great flexibility in defining function spaces with frequency-localized norms. In our setting, because the low- and high-frequency regimes exhibit different dissipative behaviors, we employ the hybrid Besov spaces introduced by Danchin \cite{refp7}, in which different regularity indices are assigned to the two frequency regimes.

\begin{definition}
	Given $s,s'\in\mathbb R$ and a fixed $k_0\in\mathbb Z$, define
	\begin{align*}
	\|u\|_{\dot B_{2,1}^{s',s}} &:= \sum_{k\le k_0}2^{k s'}\|\dot{\Delta}_k u\|_{L^2}
	+\sum_{k>k_0}2^{k s}\|\dot{\Delta}_k u\|_{L^2}\\
	&=:\|u\|_{\dot B_{2,1}^{s'}}^l+\|u\|_{\dot B_{2,1}^s}^h.
	\end{align*}
	The space $\dot B_{2,1}^{s',s}$ is the set of $u\in S'_h$ with finite norm.
\end{definition}

\begin{remark}\label{Re2.1}
	When $s'=s$, this reduces to the usual $\dot B_{2,1}^{s}$. If $s'<s$, we have $\|u\|_{\dot B_{2,1}^{s',s}}\approx \|u\|_{\dot B_{2,1}^{s'}}+\|u\|_{\dot B_{2,1}^{s}}$, so the hybrid space coincides with the intersection $\dot B_{2,1}^{s'}\cap\dot B_{2,1}^{s}$. For $s'>s$, the norm is no larger than either of the two usual norms.
\end{remark}

A standard consequence of the spectral support of $\varphi$ is Bernstein's inequality, which we shall use repeatedly.

\begin{lemma} \label{bes1.1}
	For any $f\in S'$, with $\Lambda:=\sqrt{-\Delta}$, one has
	\[
	\frac34\,2^k \|\dot{\Delta}_k f\|_{L^2}
	\le \|\Lambda \dot{\Delta}_k f\|_{L^2}
	\le \frac83\,2^k \|\dot{\Delta}_k f\|_{L^2}.
	\]
\end{lemma}

Next, we recall Bony's decomposition \cite{refp2}, which is a standard tool for nonlinear estimates in Besov spaces. The paraproduct of $f$ and $g$ is
\[
\dot T_f g := \sum_{k\in\mathbb Z}\dot S_{k-1}f\,\dot{\Delta}_k g,
\]
and the remainder is
\[
\dot R(f,g) := \sum_{k\ge -1}\dot{\Delta}_k f\,\widetilde{\Delta}_k g,
\quad\text{with } \widetilde{\Delta}_k g:=(\dot{\Delta}_{k-1}+\dot{\Delta}_k+\dot{\Delta}_{k+1})g.
\]
Then $fg$ admits the decomposition
\[
fg = \dot T_f g+\dot T_g f+\dot R(f,g) = \dot T_f g+\dot T'_g f,
\]
where $\dot T'_g f:=\dot T_g f+\dot R(f,g)$.

In this paper we also need a time-weighted version of the hybrid norm, generalising the setting of Paicu and Zhang \cite{refp22}.

\begin{definition}
	Let $f\in L^1_{\rm loc}(\mathbb R^+)$ be nonnegative, and let $s,s'\in\mathbb R$. Define
	\[
	\|u\|_{L_{t,f}^1(\dot B_{2,1}^{s',s})}
	:= \int_0^t f(t')\,\|u(t')\|_{\dot B_{2,1}^{s',s}}\,\dif t'.
	\]
\end{definition}

Finally, we introduce the analytic weight. Let $\theta(t)\in C^1[0,\infty)$ be nondecreasing with $\theta(0)=0$, and define
\[
\Phi(t,\xi):=(2\lambda_0-\lambda\theta(t))|\xi|,
\]
where $\lambda_0>0$ is fixed and $\lambda$ is a large postive constant to be chosen later. For a function $f$, we set
\begin{equation}\label{definef}
	f_\Phi(t,x):=\mathcal F^{-1}_{\xi\to x}\bigl(e^{\Phi(t,\xi)}\hat f(t,\xi)\bigr)=e^{\Phi(t,\Lambda)}f(t,x).
\end{equation}
In particular, $f_\Phi(0,x)=e^{2\lambda_0\Lambda}f(0,x)$, and
\[
\dot{\Delta}_k f_\Phi(t,x)=\mathcal F^{-1}_{\xi\to x}\bigl(\varphi(2^{-k}\xi)e^{\Phi(t,\xi)}\hat f(t,\xi)\bigr).
\]
For brevity we shall write $f_{\Phi,k}:=\dot{\Delta}_k f_\Phi$.

To handle the nonlinear terms in system \eqref{ysys}, we need to estimate products of functions in the analytic weighted spaces. For this purpose, we recall the following product estimates, whose proof can be found in \cite{pan} (see also Appendix A of \cite{pan} or Lemma 6.1 of \cite{refp30}).
\begin{lemma}\label{fge}
	Let $f_{\Phi}$ be defined as in \eqref{definef}. For $s\in(-\frac{3}{2}, \frac{3}{2}]$ and $(f_{\Phi}, g_{\Phi})\in \dot{B}^{\frac{3}{2}}_{2,1}\times\dot{B}^s_{2,1}$, there exists a positive constant $C$, depending on $s$, such that the following product estimates holds:
	\begin{equation}\label{fgbe}
		\|(fg)_{\Phi}\|_{\dot{B}^s_{2,1}}\le C\|f_{\Phi}\|_{\dot{B}^{\frac{3}{2}}_{2,1}}\|g_{\Phi}\|_{\dot{B}^s_{2,1}}.
	\end{equation}
	For $s\geq \frac{3}{2}$ and $(f_{\Phi}, g_{\Phi})\in \dot{B}^{\frac{3}{2}}_{2,1}\cap\dot{B}^s_{2,1}$, there exists a positive constant $C$, depending on $s$, such that the following product estimates holds:
	\begin{equation}\label{sdyn2}
			\|(fg)_{\Phi}\|_{\dot{B}^s_{2,1}}\le C\left(\|f_{\Phi}\|_{\dot{B}^{\frac{3}{2}}_{2,1}}\|g_{\Phi}\|_{\dot{B}^s_{2,1}}+\|f_{\Phi}\|_{\dot{B}^{s}_{2,1}}\|g_{\Phi}\|_{\dot{B}^{\frac{3}{2}}_{2,1}}\right).
	\end{equation}
	In particular, for $s=\frac{3}{2}$,  \eqref{fgbe} and induction on $k$ yield
	\begin{equation}\label{fgbek}
		\|(f^k)_{\Phi}\|_{\dot{B}^s_{2,1}}\le C^k\|f_{\Phi}\|^k_{\dot{B}^{\frac{3}{2}}_{2,1}},
	\end{equation}
	for any positive integer $k$.
\end{lemma}

\section{A priori estimates  and proof of Theorem \ref{th1.1}}
\subsection{A priori estimates for the linearized system} In this part, we linearize the system \eqref{ysys} and then derive an a priori estimate for it. Let $a=\rho-1$, by \eqref{ysys}, we get
\begin{equation}\label{lsys}
	\begin{cases}
		a_t +\ddiv\mathbf{u}=F,\\
		 \mathbf{u}_t+\nabla a-\ddiv\mathbf{S} =G,\\
		\tau\mathbf{S}_t+\mathbf{S}-2\mu\mathbf{D}(\mathbf{u})=\tau H,
	\end{cases}
\end{equation}
where 
\begin{align*}
&	F=-\nabla\cdot(a\mathbf{u}),\\
&	G=-\mathbf{u}\cdot\nabla\mathbf{u}+[1-\gamma A(1+a)^{\gamma -2}]\nabla a+\frac{1}{1+a}\mathbf{S}\cdot\nabla a,\\
&	H=-\left(\mathbf{u}\cdot\nabla \mathbf{S}-\nabla\mathbf{u}\cdot\mathbf{S}-\mathbf{S}\cdot(\nabla\mathbf{u})^T\right).
\end{align*}
We now establish the a priori estimate for the  above system. For the convenience of subsequent calculations, we set $A=\frac{1}{\gamma}$.

\begin{prop}\label{prop1}
	Assume $0\leq	\tau\leq\min\{1, \mu\}$ and $(a, \mathbf{u}, \mathbf{S})$ is a solution of the linearized system \eqref{lsys}, with the initial data $e^{2\lambda_0\Lambda}(a_0, \mathbf{u}_0, \sqrt{\tau}\mathbf{S}_0)\in\dot{B}_{2,1}^{\frac{1}{2},\frac{3}{2}}$. Then there exists a positive constant $C$ independent of $\tau$, such that for any $t>0$, there holds
	\begin{equation}\label{priori}
		\begin{aligned}
		&\|(a_{\Phi},\mathbf{u}_{\Phi},\sqrt{{\tau}}\mathbf{S}_{\Phi})\|_{\dot{B}^{\frac{1}{2},\frac{3}{2}}_{2,1}}+\lambda\int_{0}^{t}\dot{\theta}(t)\|(a_{\Phi},\mathbf{u}_{\Phi},\sqrt{{\tau}}\mathbf{S}_{\Phi})\|_{\dot{B}^{\frac{3}{2},\frac{5}{2}}_{2,1}}\dif t+\int_{0}^{t}\|(a_{\Phi},\mathbf{u}_{\Phi},\mathbf{S}_{\Phi})\|_{\dot{B}^{\frac{3}{2}}_{2,1}}\dif t\\
		&\le C(\|e^{2\lambda_0\Lambda}(a_0,\mathbf{u}_0,\sqrt{{\tau}}\mathbf{S}_0)\|_{\dot{B}^{\frac{1}{2},\frac{3}{2}}_{2,1}}+\int_0^t \|(F_{\Phi},G_{\Phi},\sqrt{\tau}H_{\Phi})\|_{\dot{B}^{\frac{1}{2},\frac{3}{2}}_{2,1}}\dif t).
		\end{aligned}
	\end{equation} 
\end{prop}

\begin{proof}
From \eqref{lsys}, we derive $(a_{\Phi,k}, \mathbf{u}_{\Phi,k}, \mathbf{S}_{\Phi,k})$ satisfying the following system
	\begin{equation}\label{littlewood}
	\begin{cases}
	\partial_t a_{\Phi,k}+\lambda\dot{\theta}(t)\Lambda a_{\Phi,k} +\ddiv\mathbf{u}_{\Phi,k}
	=F_{\Phi,k},\\
	\partial_t	\mathbf{u}_{\Phi,k}+\lambda\dot{\theta}(t)\Lambda\mathbf{u}_{\Phi,k}+\nabla a_{\Phi,k}-\ddiv\mathbf{S}_{\Phi,k}
	=G_{\Phi,k},\\
	\tau\partial_t	\mathbf{S}_{\Phi,k}+\tau\lambda\dot{\theta}(t)\Lambda\mathbf{S}_{\Phi,k}+\mathbf{S}_{\Phi,k}
	-2\mu\mathbf{D}(\mathbf{u})_{\Phi,k}
	=\tau H_{\Phi,k}.
		\end{cases}
	\end{equation}
Multiplying the above equations by $2a_{\Phi,k}$, $2\mathbf{u}_{\Phi,k}$ and  $\frac{1}{\mu}\mathbf{S}_{\Phi,k}$, respectively, and integrating over $\mathbb{R}^3$, we have
\begin{align*}
\frac{1}{2}&\frac{\dif }{\dif t}\left(2\|a_{\Phi,k}\|_{L^2}^2+2\|\mathbf{u}_{\Phi,k}\|_{L^2}^2+\frac{\tau}{\mu}\|\mathbf{S}_{\Phi,k}\|_{L^2}^2   \right)
\\
&+\lambda\dot{\theta}(t)\left(2\|\Lambda^{\frac{1}{2}}a_{\Phi,k}\|_{L^2}^2+2\|\Lambda^{\frac{1}{2}}\mathbf{u}_{\Phi,k}\|_{L^2}^2+\frac{\tau}{\mu}\|\Lambda^{\frac{1}{2}}\mathbf{S}_{\Phi,k}\|_{L^2}^2\right)+\frac{1}{\mu}\|\mathbf{S}_{\Phi,k}\|_{L^2}^2\\
&+2\langle \ddiv\mathbf{u}_{\Phi,k},a_{\Phi,k}\rangle+2\langle \nabla a_{\Phi,k},\mathbf{u}_{\Phi,k}\rangle-2\langle \ddiv\mathbf{S}_{\Phi,k},\mathbf{u}_{\Phi,k}\rangle-2\langle \mathbf{D}(\mathbf{u})_{\Phi,k},\mathbf{S}_{\Phi,k}\rangle\\
=&2\langle F_{\Phi,k},a_{\Phi,k}\rangle+2\langle G_{\Phi,k},\mathbf{u}_{\Phi,k}\rangle+\frac{\tau}{\mu}\langle H_{\Phi,k},\mathbf{S}_{\Phi,k}\rangle.
\end{align*}
 Using the symmetric property of $\mathbf{S}$ and integrating by parts, one get
\begin{equation}\label{l2e}
\begin{aligned}
	\frac{1}{2}&\frac{\dif }{\dif t}\left(2\|a_{\Phi,k}\|_{L^2}^2+2\|\mathbf{u}_{\Phi,k}\|_{L^2}^2+\frac{\tau}{\mu}\|\mathbf{S}_{\Phi,k}\|_{L^2}^2   \right)
	\\
	&+\lambda\dot{\theta}(t)\left(2\|\Lambda^{\frac{1}{2}}a_{\Phi,k}\|_{L^2}^2+2\|\Lambda^{\frac{1}{2}}\mathbf{u}_{\Phi,k}\|_{L^2}^2+\frac{\tau}{\mu}\|\Lambda^{\frac{1}{2}}\mathbf{S}_{\Phi,k}\|_{L^2}^2\right)+\frac{1}{\mu}\|\mathbf{S}_{\Phi,k}\|_{L^2}^2\\
	=&2\langle F_{\Phi,k},a_{\Phi,k}\rangle+2\langle G_{\Phi,k},\mathbf{u}_{\Phi,k}\rangle+\frac{\tau}{\mu}\langle H_{\Phi,k},\mathbf{S}_{\Phi,k}\rangle.
\end{aligned}
\end{equation}
Next, we do the dissipation estimates for $(a_{\Phi,k}, \mathbf{u}_{\Phi,k})$ to close the energy. To do this, we need to do the estimates in high-frequencies and low-frequencies separately. 

$\textbf{Estimate in high-frequencies.}$ Applying $\Lambda^{-1}\nabla$ to $\eqref{littlewood}_1$, $\Lambda^{-1}$ to $\eqref{littlewood}_2$ and $\Lambda^{-1}\ddiv$ to $\eqref{littlewood}_3$, we derive that
	\begin{equation}\label{hfg}
	\begin{cases}
		\partial_t\Lambda^{-1}\nabla a_{\Phi,k}+\lambda\dot{\theta}(t)\nabla a_{\Phi,k} +\Lambda^{-1}\nabla\ddiv\mathbf{u}_{\Phi,k}
		=\Lambda^{-1}\nabla F_{\Phi,k},\\
		\partial_t\Lambda^{-1}	\mathbf{u}_{\Phi,k}+\lambda\dot{\theta}(t)\mathbf{u}_{\Phi,k}+\Lambda^{-1}\nabla a_{\Phi,k}-\Lambda^{-1}\ddiv\mathbf{S}_{\Phi,k}
		=\Lambda^{-1}G_{\Phi,k},\\
		\tau\partial_t\Lambda^{-1}\ddiv	\mathbf{S}_{\Phi,k}+\tau\lambda\dot{\theta}(t)\ddiv\mathbf{S}_{\Phi,k}+\Lambda^{-1}\ddiv\mathbf{S}_{\Phi,k}
		+\mu(\Lambda-\nabla\Lambda^{-1}\ddiv)\mathbf{u}_{\Phi,k}
		=\tau\Lambda^{-1}\ddiv H_{\Phi,k}.
	\end{cases}
\end{equation}
Testing $\eqref{hfg}_1$ against \(\Lambda^{-1}\mathbf{u}_{\Phi,k}\) and $\eqref{hfg}_2$ against \(\Lambda^{-1}\nabla a_{\Phi,k}\), and then integrating by parts, we obtain
\begin{equation}\label{hfhs1}
	\begin{aligned}
	\partial_t&\langle\Lambda^{-1}\nabla a_{\Phi,k},\Lambda^{-1}\mathbf{u}_{\Phi,k}\rangle+2\lambda\dot{\theta}(t)\langle\nabla a_{\Phi,k},\Lambda^{-1}\mathbf{u}_{\Phi,k}\rangle-\|\Lambda^{-1}\ddiv\mathbf{u}_{\Phi,k}\|_{L^2}^2+\|a_{\Phi,k}\|_{L^2}^2\\
	& -\langle\Lambda^{-1}\ddiv\mathbf{S}_{\Phi,k},\Lambda^{-1}\nabla a_{\Phi,k} \rangle\\
	=&\langle \Lambda^{-1}\nabla F_{\Phi,k} ,\Lambda^{-1}\mathbf{u}_{\Phi,k}\rangle+\langle \Lambda^{-1} G_{\Phi,k},\Lambda^{-1}\nabla a_{\Phi,k}\rangle,
	\end{aligned}
\end{equation}
where 
\[
\langle \Lambda^{-1}\nabla \ddiv\mathbf{u}_{\Phi,k} ,\Lambda^{-1}\mathbf{u}_{\Phi,k}\rangle=-\|\Lambda^{-1}\ddiv \mathbf{u}_{\Phi,k}\|_{L^2}^2.
\]
Multiplying $\eqref{hfg}_2$ by \(\frac{\tau}{\mu}\Lambda^{-1}\ddiv\mathbf{S}_{\Phi,k}\)
and $\eqref{hfg}_3$ by \(\frac{1}{\mu}\Lambda^{-1}\mathbf{u}_{\Phi,k}\), followed
by integration by parts, we get
\begin{equation}\label{hfhs2}
	\begin{aligned}
		\frac{\tau}{\mu}\partial_t&\langle\Lambda^{-1}\mathbf{u}_{\Phi,k},\Lambda^{-1}\ddiv\mathbf{S}_{\Phi,k}\rangle+2\frac{\tau}{\mu}\lambda\dot{\theta}(t)\langle\Lambda^{-1}\mathbf{u}_{\Phi,k},\ddiv\mathbf{S}_{\Phi,k}\rangle+\frac{\tau}{\mu}\langle \Lambda^{-1}\nabla a_{\Phi,k}, \Lambda^{-1}\ddiv\mathbf{S}_{\Phi,k}\rangle\\
		&-\frac{\tau}{\mu}\|\Lambda^{-1}\ddiv\mathbf{S}_{\Phi,k}\|_{L^2}^2+\frac{1}{\mu}\langle \Lambda^{-1}\ddiv\mathbf{S}_{\Phi,k}, \Lambda^{-1}\mathbf{u}_{\Phi,k}\rangle+\|\mathbf{u}_{\Phi,k}\|_{L^2}^2+\|\Lambda^{-1}\ddiv\mathbf{u}_{\Phi,k}\|_{L^2}^2\\
		=&\frac{\tau}{\mu}\langle \Lambda^{-1}G_{\Phi,k} ,\Lambda^{-1}\ddiv\mathbf{S}_{\Phi,k}\rangle+\frac{\tau}{\mu}\langle \Lambda^{-1}\ddiv H_{\Phi,k},\Lambda^{-1} \mathbf{u}_{\Phi,k}\rangle.
	\end{aligned}
\end{equation}
Now, for $k\ge k_0$, we define
\begin{align*}
\widetilde{E}^2_k=2\|a_{\Phi,k}\|_{L^2}^2+2\|\mathbf{u}_{\Phi,k}\|_{L^2}^2+\frac{\tau}{\mu}\|\mathbf{S}_{\Phi,k}\|_{L^2}^2+\langle\Lambda^{-1}\nabla a_{\Phi,k},\Lambda^{-1}\mathbf{u}_{\Phi,k}\rangle+\langle\Lambda^{-1}\mathbf{u}_{\Phi,k},\frac{\tau}{\mu}\Lambda^{-1}\ddiv\mathbf{S}_{\Phi,k}\rangle, 
\end{align*}
\[
E^2_k=\|a_{\Phi,k}\|_{L^2}^2+\|\mathbf{u}_{\Phi,k}\|_{L^2}^2+\frac{\tau}{\mu}\|\mathbf{S}_{\Phi,k}\|_{L^2}^2,
\]
and
\[
D^2_k=\|a_{\Phi,k}\|_{L^2}^2+\|\mathbf{u}_{\Phi,k}\|_{L^2}^2+\|\mathbf{S}_{\Phi,k}\|_{L^2}^2.
\]
Choosing $k_0=3$, using Bernstein's inequality (Lemma \ref{bes1.1}), it holds
\[
\frac{1}{2}E_k^2\le \widetilde{E}^2_k \le 3E^2_k.
\]
Combining  \eqref{l2e}, \eqref{hfhs1} and \eqref{hfhs2}, and then using H\"older's inequality, we conclude that 
\[
\frac{\dif}{\dif t}\widetilde{E}^2_k+\frac{1}{4}\lambda\dot{\theta}(t)2^k\widetilde{E}^2_k+\frac{1}{8}D_k^2 \le C \|(F_{\Phi,k},G_{\Phi,k},\sqrt{\frac{\tau}{\mu}}H_{\Phi,k})\|_{L^2}\widetilde{E}_k.
\]
Here and in what follows, $C$ denotes a universal constant independent of $\tau$.
Noting that ${E}_k\le D_k$ by using the assumption $\tau \le \mu$, we obtain that
\begin{equation}\label{gpfd}
\frac{\dif}{\dif t}\widetilde E_k+\lambda\dot{\theta}(t)2^k \widetilde E_k+D_k \le C \|(F_{\Phi,k},G_{\Phi,k},\sqrt{\frac{\tau}{\mu}}H_{\Phi,k})\|_{L^2}.
\end{equation}
Multiplying \eqref{gpfd} by $2^{\frac{3}{2}k}$, integrating over $(0,t)$ and summing over $k\ge k_0$, we obtain
\begin{equation}\label{gpbesov}
	\begin{aligned}
&\|(a_{\Phi},\mathbf{u}_{\Phi},\sqrt{\frac{\tau}{\mu}}\mathbf{S}_{\Phi})\|^h_{\dot{B}^{\frac{3}{2}}_{2,1}}+\lambda\int_{0}^{t}\dot{\theta}(t)\|(a_{\Phi},\mathbf{u}_{\Phi},\sqrt{\frac{\tau}{\mu}}\mathbf{S}_{\Phi})\|^h_{\dot{B}^{\frac{5}{2}}_{2,1}}\dif t\quad+\int_{0}^{t}\|(a_{\Phi},\mathbf{u}_{\Phi},\mathbf{S}_{\Phi})\|^h_{\dot{B}^{\frac{3}{2}}_{2,1}}\dif t\\
&\le C ( \|(a_{\Phi},\mathbf{u}_{\Phi},\sqrt{\frac{\tau}{\mu}}\mathbf{S}_{\Phi})(0)\|^h_{\dot{B}^{\frac{3}{2}}_{2,1}}+\int_0^t \|(F_{\Phi},G_{\Phi},\sqrt{\frac{\tau}{\mu}}H_{\Phi})\|^h_{\dot{B}^{\frac{3}{2}}_{2,1}}\dif t).
	\end{aligned}
\end{equation}

$\textbf{Estimate  in low-frequencies.}$ Applying $\nabla$ to $\eqref{littlewood}_1$ and $\ddiv$ to $\eqref{littlewood}_3$, we derive that
	\begin{equation}\label{lfg}
	\begin{cases}
		\partial_t\nabla a_{\Phi,k}+\lambda\dot{\theta}(t)\nabla\Lambda a_{\Phi,k} +\nabla\ddiv\mathbf{u}_{\Phi,k}
		=\nabla F_{\Phi,k},\\
		\partial_t	\mathbf{u}_{\Phi,k}+\lambda\dot{\theta}(t)\Lambda\mathbf{u}_{\Phi,k}+\nabla a_{\Phi,k}-\ddiv\mathbf{S}_{\Phi,k}
		=G_{\Phi,k},\\
		\tau\partial_t\ddiv	\mathbf{S}_{\Phi,k}+\tau\lambda\dot{\theta}(t)\ddiv\Lambda\mathbf{S}_{\Phi,k}+\ddiv\mathbf{S}_{\Phi,k}
		-\mu(\Delta+\nabla\ddiv)\mathbf{u}_{\Phi,k}
		=\tau \ddiv H_{\Phi,k}.
	\end{cases}
\end{equation}
Testing $\eqref{lfg}_1$ against \(\mathbf{u}_{\Phi,k}\) and $\eqref{lfg}_2$
against \(\nabla a_{\Phi,k}\), and then integrating by parts, we arrive at
\begin{equation}\label{lfhs1}
	\begin{aligned}
		\partial_t&\langle\nabla a_{\Phi,k},\mathbf{u}_{\Phi,k}\rangle+2\lambda\dot{\theta}(t)\langle\nabla \Lambda a_{\Phi,k},\mathbf{u}_{\Phi,k}\rangle-\|\ddiv\mathbf{u}_{\Phi,k}\|_{L^2}^2+\|\nabla a_{\Phi,k}\|_{L^2}^2
		 -\langle\ddiv\mathbf{S}_{\Phi,k},\nabla a_{\Phi,k} \rangle\\
		=&\langle \nabla F_{\Phi,k} ,\mathbf{u}_{\Phi,k}\rangle+\langle  G_{\Phi,k},\nabla a_{\Phi,k}\rangle,
	\end{aligned}
\end{equation}
where 
\[
\langle \nabla \ddiv \mathbf{u}_{\Phi,k} ,\mathbf{u}_{\Phi,k}\rangle=-\|\ddiv\mathbf{u}_{\Phi,k}\|_{L^2}^2.
\]
Multiplying $\eqref{lfg}_2$ by $\frac{\tau}{\mu}\ddiv\mathbf{S}_{\Phi,k}$ and $\eqref{lfg}_3$ by $\frac{1}{\mu} \mathbf{u}_{\Phi,k}$, and integrating by parts, we get 
\begin{equation}\label{lfhs2}
	\begin{aligned}
		\frac{\tau}{\mu}\partial_t&\langle\mathbf{u}_{\Phi,k},\ddiv\mathbf{S}_{\Phi,k}\rangle+2\frac{\tau}{\mu}\lambda\dot{\theta}(t)\langle\Lambda\mathbf{u}_{\Phi,k},\ddiv\mathbf{S}_{\Phi,k}\rangle+\frac{\tau}{\mu}\langle \nabla a_{\Phi,k}, \ddiv\mathbf{S}_{\Phi,k}\rangle\\
		&-\frac{\tau}{\mu}\|\ddiv\mathbf{S}_{\Phi,k}\|_{L^2}^2+\frac{1}{\mu}\langle \ddiv\mathbf{S}_{\Phi,k}, \mathbf{u}_{\Phi,k}\rangle+\|\nabla \mathbf{u}_{\Phi,k}\|_{L^2}^2+\|\ddiv\mathbf{u}_{\Phi,k}\|_{L^2}^2\\
		=&\frac{\tau}{\mu}\langle G_{\Phi,k} ,\ddiv\mathbf{T}_{\Phi,k}\rangle+\frac{\tau}{\mu}\langle \ddiv H_{\Phi,k}, \mathbf{u}_{\Phi,k}\rangle.
	\end{aligned}
\end{equation}
For $k\le k_0$, define $$\widetilde{E}^2_k=2\|a_{\Phi,k}\|_{L^2}^2+2\|\mathbf{u}_{\Phi,k}\|_{L^2}^2+\frac{\tau}{\mu}\|\mathbf{S}_{\Phi,k}\|_{L^2}^2+\frac{3}{8}2^{-k}\langle\nabla a_{\Phi,k},\mathbf{u}_{\Phi,k}\rangle+\frac{3}{8}2^{-k}\langle\mathbf{u}_{\Phi,k},\frac{\tau}{\mu}\ddiv\mathbf{S}_{\Phi,k}\rangle,  $$
\[
E^2_k=\|a_{\Phi,k}\|_{L^2}^2+\|\mathbf{u}_{\Phi,k}\|_{L^2}^2+\frac{\tau}{\mu}\|\mathbf{S}_{\Phi,k}\|_{L^2}^2,
\]
and
\[
D^2_k=\|a_{\Phi,k}\|_{L^2}^2+\|\mathbf{u}_{\Phi,k}\|_{L^2}^2+\|\mathbf{S}_{\Phi,k}\|_{L^2}^2.
\]
Similarly, by using Bernstein's inequality, it holds
\[
\frac{1}{2}E_k^2\le \widetilde{E}^2_k \le 3E^2_k.
\]
Multiplying both \eqref{lfhs1} and \eqref{lfhs2} by $\frac{3}{8}2^{-k}$, and then adding \eqref{l2e} to the sum, followed by applying H\"older's inequality, we conclude that 
\[
\frac{\dif}{\dif t}\widetilde{E}^2_k+\frac{1}{4}\lambda\dot{\theta}(t)2^k\widetilde{E}^2_k+\frac{1}{8}2^kD_k^2 \le C \|(F_{\Phi,k},G_{\Phi,k},\sqrt{\frac{\tau}{\mu}}H_{\Phi,k})\|_{L^2}\widetilde{E}_k.
\]
Noting that $E_k\le D_k$, we obtain that
\begin{equation}\label{dpfd}
	\frac{\dif}{\dif t} \widetilde E_k+\lambda\dot{\theta}(t)2^k \widetilde E_k+2^kD_k \le C  \|(F_{\Phi,k},G_{\Phi,k},\sqrt{\frac{\tau}{\mu}}H_{\Phi,k})\|_{L^2}.
\end{equation}
Multiplying \eqref{dpfd} by $2^{\frac{k}{2}}$, integrating over $(0,t)$ and summing over $k\le k_0$, we obtain
\begin{equation}\label{dpbesov}
	\begin{aligned}
		&\|(a_{\Phi},\mathbf{u}_{\Phi},\sqrt{\frac{\tau}{\mu}}\mathbf{S}_{\Phi})\|^l_{\dot{B}^{\frac{1}{2}}_{2,1}}+\lambda\int_{0}^{t}\dot{\theta}(t)\|(a_{\Phi},\mathbf{u}_{\Phi},\sqrt{\frac{\tau}{\mu}}\mathbf{S}_{\Phi})\|^l_{\dot{B}^{\frac{3}{2}}_{2,1}}\dif t\\
		&\quad+\int_{0}^{t}\|(a_{\Phi},\mathbf{u}_{\Phi},\mathbf{S}_{\Phi})\|^l_{\dot{B}^{\frac{3}{2}}_{2,1}}\dif t\\
		&\le C( \|(a_{\Phi},\mathbf{u}_{\Phi},\sqrt{\frac{\tau}{\mu}}\mathbf{S}_{\Phi})(0)\|^l_{\dot{B}^{\frac{1}{2}}_{2,1}}+\int_0^t \|(F_{\Phi},G_{\Phi},\sqrt{\frac{\tau}{\mu}}H_{\Phi})\|^l_{\dot{B}^{\frac{1}{2}}_{2,1}}\dif t).
	\end{aligned}
\end{equation}

Therefore, combining \eqref{gpbesov} and \eqref{dpbesov}, we conclude that 
\begin{equation}\label{besovfs}
	\begin{aligned}
		&\|(a_{\Phi},\mathbf{u}_{\Phi},\sqrt{\frac{\tau}{\mu}}\mathbf{S}_{\Phi})\|_{\dot{B}^{\frac{1}{2},\frac{3}{2}}_{2,1}}+\lambda\int_{0}^{t}\dot{\theta}(t)\|(a_{\Phi},\mathbf{u}_{\Phi},\sqrt{\frac{\tau}{\mu}}\mathbf{S}_{\Phi})\|_{\dot{B}^{\frac{3}{2},\frac{5}{2}}_{2,1}}\dif t\\
		&\quad+\int_{0}^{t}\|(a_{\Phi},\mathbf{u}_{\Phi},\mathbf{S}_{\Phi})\|_{\dot{B}^{\frac{3}{2}}_{2,1}}\dif t\\
		&\le C( \|e^{2\lambda\Lambda}(a_0,\mathbf{u}_0,\sqrt{\frac{\tau}{\mu}}\mathbf{S}_0)\|_{\dot{B}^{\frac{1}{2},\frac{3}{2}}_{2,1}}+\int_0^t \|(F_{\Phi},G_{\Phi},\sqrt{\frac{\tau}{\mu}}H_{\Phi})\|_{\dot{B}^{\frac{1}{2},\frac{3}{2}}_{2,1}}\dif t).
	\end{aligned}
\end{equation}
This finishes the proof of Proposition \ref{prop1}.

\end{proof}

To close the estimate in Proposition \ref{prop1}, we need to bound the nonlinear terms \(F, G\), and \(H\) in the analytic weighted spaces. This is achieved by the following lemma.

\begin{lemma}\label{efa}
	There exist constants $\delta \in(0,\frac{1}{2})$ and $C$ independent of $\tau$ such that, if for any $t>0$,
	\begin{align}\label{APa}
	\|a_\Phi(\cdot, t)\|_{\dot{B}^{\frac{3}{2}}_{2,1}}\le \delta
	\end{align}
	and 
	\begin{align}\label{APb}
	\lambda_0 |\xi| \le \Phi(t,\xi) \le 2 \lambda_0 |\xi|, 
	\end{align}
	then for $s\in(-\frac{3}{2}, \frac{3}{2}]$, the following estimates hold
	\begin{equation}\label{tesfa}
		\left	\|\left(F_{\Phi}, G_{\Phi}, \sqrt{\tau}H_{\Phi}\right)\right\|_{\dot{B}^{s}_{2,1}}\le C\|(a_{\Phi},\mathbf{u}_{\Phi}, \mathbf{S}_{\Phi})\|_{\dot{B}^{\frac{3}{2}}_{2,1}}\|(a_{\Phi},\mathbf{u}_{\Phi}, \sqrt{\tau}\mathbf{S}_{\Phi})\|_{\dot{B}^{s+1}_{2,1}}.
	\end{equation} 
\end{lemma}
\begin{proof}
Under the a priori assumption \eqref{APa} and \eqref{APb}, we have the following fact
	\[
	\left	\|\left(1-(1+a)^{\gamma-2}\right)_{\Phi}\right\|_{\dot{B}^{\frac{3}{2}}_{2,1}}\le C\|a_{\Phi}\|_{\dot{B}^{\frac{3}{2}}_{2,1}},\, 	\left\|\left(\frac{a}{1+a}\right)_{\Phi}\right\|_{\dot{B}^{\frac{3}{2}}_{2,1}}\le C\|a_{\Phi}\|_{\dot{B}^{\frac{3}{2}}_{2,1}},\,
	\left	\|\left(\frac{1}{1+a}\right)_{\Phi}\right\|_{\dot{B}^{\frac{3}{2}}_{2,1}} \leq C.
	\]
	For the proof of the first two estimates, see Lemma 2.9 in \cite{pan}; we now prove the third. 
	
	 First, by using a priori assumption \eqref{APa} and \eqref{APb}, we have
	\begin{align*}
	\|a(\cdot, t)\|_{\dot B_{2,1}^\frac{3}{2}} =\|e^{-\Phi(t, \Lambda)} a_\Phi\|_{\dot B_{2,1}^\frac{3}{2}} =\|\mathcal F^{-1} (e^{-\Phi(t,\xi)}) \ast a_{\Phi}\|_{\dot B_{2,1}^\frac{3}{2}}  
	\le C \|a_\Phi(\cdot, t)\|_{\dot{B}^{\frac{3}{2}}_{2,1}} \le C \delta.
	\end{align*}
	 By choosing $\delta$ sufficient small and using embedding theorem, we have $$\|a(\cdot, t)\|_{L^\infty}\le C\|a(\cdot, t)\|_{\dot B_{2,1}^\frac{3}{2}} \le 1.$$  Then, using Taylor expansion, we get
	\[
	\frac{1}{1+a}=\sum_{k=0}^{\infty}(-1)^k a^k.
	\]
	Finally, using \eqref{fgbek} and $\delta<\frac{1}{2}$, we derive that
	\[
	\left\|\left(\frac{1}{1+a}\right)_{\Phi}\right\|_{\dot{B}^{\frac{3}{2}}_{2,1}}\le  \sum_{k=0}^{\infty}\|(a^k)_{\Phi}\|_{\dot{B}^{\frac{3}{2}}_{2,1}}\le  C\sum_{k=0}^{\infty}\|a_{\Phi}\|_{\dot{B}^{\frac{3}{2}}_{2,1}}^k\le \frac{C}{1-\|a_{\Phi}\|_{\dot{B}^{\frac{3}{2}}_{2,1}}}\leq C.
	\]
Therefore, from the definition of $F,G,H$ and using Lemma \ref{fge}, we derive that
\begin{align*}
\|(F_{\Phi},G_{\Phi},\sqrt{\tau}H_{\Phi})\|_{\dot{B}^{s}_{2,1}}&\leq C\|( a_{\Phi},\mathbf{u}_{\Phi},\mathbf{S}_{\Phi})\|_{\dot{B}^{\frac{3}{2}}_{2,1}}\|(\nabla a_{\Phi},\nabla\mathbf{u}_{\Phi},\sqrt{{\tau}}\nabla\mathbf{S}_{\Phi})\|_{\dot{B}^{ s}_{2,1}}  \\
&\leq C\|(a_{\Phi},\mathbf{u}_{\Phi},\mathbf{S}_{\Phi})\|_{\dot{B}^{\frac{3}{2}}_{2,1}}\|(a_{\Phi},\mathbf{u}_{\Phi},\sqrt{{\tau}}\mathbf{S}_{\Phi})\|_{\dot{B}^{s+1}_{2,1}}.
\end{align*}
Thus, the proof of this lemma is completed.
\end{proof}

\subsection{Proof of Theorem \ref{th1.1}}

We now prove Theorem~\ref{th1.1} using Proposition~\ref{prop1} and Lemma~\ref{efa}.  Under the a priori assumption \eqref{APa} and \eqref{APb},  by substituting \eqref{tesfa} into \eqref{priori} and noting Remark \ref{Re2.1}, we obtain the  following key estimate
\begin{equation}\label{priories}
	\begin{aligned}
		&\|(a_{\Phi},\mathbf{u}_{\Phi},{\sqrt{\tau}}\mathbf{S}_{\Phi})\|_{\dot{B}^{\frac{1}{2},\frac{3}{2}}_{2,1}}
		+\lambda\|(a_{\Phi},\mathbf{u}_{\Phi},\sqrt{{\tau}}\mathbf{S}_{\Phi})\|_{L^1_{t,\dot{\theta}}(\dot{B}^{\frac{3}{2},\frac{5}{2}}_{2,1})}
		+\|(a_{\Phi},\mathbf{u}_{\Phi},\mathbf{S}_{\Phi})\|_{L^1_t(\dot{B}^{\frac{3}{2}}_{2,1})}\\
		&\le C_0\left(\|e^{2\lambda_0\Lambda}(a_0,\mathbf{u}_0,\sqrt{{\tau}}\mathbf{S}_0)\|_{\dot{B}^{\frac{1}{2},\frac{3}{2}}_{2,1}}
		+\int_0^t \|(a_{\Phi},\mathbf{u}_{\Phi},\mathbf{S}_{\Phi})\|_{\dot{B}^{\frac{3}{2}}_{2,1}}
		\|(a_{\Phi},\mathbf{u}_{\Phi},\sqrt{{\tau}}\mathbf{S}_{\Phi})\|_{\dot{B}^{\frac{3}{2},\frac{5}{2}}_{2,1}}\dif t\right),
	\end{aligned}
\end{equation}
where $C_0$ is a positive universal constant independent of $\tau$.

To construct global solutions, we employ the standard Friedrichs approximation method (see, e.g., \cite{refz5} and Chapter~10 of \cite{refp1}). For each $k\in\mathbb N^+$, let $L^2_k(\mathbb R^3)$ denote the space of $L^2$ functions whose Fourier support lies in the annulus
\[
\mathcal C_k:=\{\xi\in\mathbb R^3: k^{-1}\le |\xi|\le k\}.
\]
The projection operator $\dot{\mathbb E}_k$ is defined by
\[
\dot{\mathbb E}_k U(x):=\mathcal F^{-1}\bigl(\mathbf 1_{\mathcal C_k}\mathcal F U(\xi)\bigr).
\]
Consider the truncated system
\begin{equation}\label{aode}
	\begin{aligned}
		\frac{\dif}{\dif t}
		&\begin{pmatrix}
			a \\
			\mathbf{u} \\
			\tau\mathbf{S}
		\end{pmatrix}
		=\dot{\mathbb E}_k
		\begin{pmatrix}
			-\ddiv \mathbf{u}-\nabla\cdot(a\mathbf{u}) \\
			-\nabla a+\ddiv\mathbf{S}+\mathbf{u}\cdot\nabla\mathbf{u}+[1-(1+a)^{\gamma-2}]\nabla a+\frac{1}{1+a}\mathbf{S}\cdot\nabla a \\
			-\mathbf{S}+2\mu\mathbf{D}(\mathbf{u})-\tau\mathbf{u}\cdot\nabla \mathbf{S}+\tau\nabla\mathbf{u}\cdot\mathbf{S}+\tau\mathbf{S}\cdot(\nabla\mathbf{u})^T
		\end{pmatrix},\\
		&	\begin{pmatrix}
			a \\
			\mathbf{u} \\
			\mathbf{S}
		\end{pmatrix}(0,x)
		=\dot{\mathbb E}_k
		\begin{pmatrix}
			a_0 \\
			\mathbf{u}_0 \\
			\mathbf{S}_0
		\end{pmatrix}(x).
	\end{aligned}
\end{equation}
Here we have set $a:=\rho-1$ as before. Provided $\|a_0\|_{\dot B^{3/2}_{2,1}}$ is sufficiently small, the quantity $1+\dot{\mathbb E}_k a_0$ remains positive for all large $k$, so the initial data in \eqref{aode} is well defined. Owing to the low-frequency cut-off, all Sobolev norms are equivalent on the range of $\dot{\mathbb E}_k$. Standard ODE theory then yields, for each $k$, a unique maximal solution
\[
(a^k,\mathbf u^k,\mathbf S^k)\in C^1\bigl([0,T_k);\tilde L^2_k(\mathbb R^3)\bigr),
\]
where the state space $\tilde L^2_k$ consists of triples  of $L^2$-integrable functions satisfying $\inf_x(1+a^k)>0$.

For such a solution we introduce the analytic radius via the ODE
\begin{equation}\label{thetak}
	\dot{\theta}_k(t)=\|(a^k_{\Phi_k}, \mathbf{u}^k_{\Phi_k}, \mathbf{S}^k_{\Phi_k})\|_{\dot{B}^{\frac{3}{2}}_{2,1}},\quad
	\theta_k(0)=0,\quad
	\Phi_k(t,\xi)=(2\lambda_0-\lambda\theta_k(t))|\xi|.
\end{equation}
Since the Fourier transform of the solution $(a^k, \mathbf{u}^k, \mathbf{S}^k)$ is compactly supported, the right-hand side of the above equation is Lipschitz in $\theta_k$, and \eqref{thetak} admits a unique solution on $[0,T_k)$.

By initial assumption \eqref{Initial}, it is clear that
\[
\|e^{2\lambda_0\Lambda}(a^k_0,\mathbf u^k_0, \sqrt{\tau} \mathbf S^k_0)\|_{\dot B^{\frac12,\frac32}_{2,1}}
\le \|e^{2\lambda_0\Lambda}(a_0,\mathbf u_0, \sqrt{\tau}\mathbf S_0)\|_{\dot B^{\frac12,\frac32}_{2,1}}
\le \varepsilon_0\lambda_0.
\]
We now define
\[
T^\ast_k:=\sup\Bigl\{t\in[0,T_k):\;
\theta_k(t)\le \frac{\lambda_0}{\lambda},\;
\|(a^k_{\Phi_k},\mathbf u^k_{\Phi_k}, \sqrt{\tau}\mathbf S^k_{\Phi_k})\|_{\dot B^{\frac12,\frac32}_{2,1}}\le M\varepsilon_0\Bigr\},
\]
with $M:=2C_0 \lambda_0$  and $\lambda$ to be fixed later. Our goal is to show that $T_k^\ast=T_k$ by a standard continuity argument, using the a priori estimate \eqref{priories} adapted to the approximate system \eqref{aode}.

First, note that the inequality $\theta_k(t)\le \frac{\lambda_0}{\lambda}$ implies $\lambda_0 |\xi| \le \Phi_k(t,\xi) \le 2 \lambda_0 |\xi|$. Moreover, we can choose $\varepsilon_0=\min\{\frac{\delta}{M}, \frac{1}{4C_0^2}\}$ such that $\|a_{\Phi_k}^k(\cdot, t)\|_{\dot B_{2,1}^\frac{3}{2}}\le \delta$.  This close the a priori assumption \eqref{APa} and \eqref{APb}. Then, applying the same energy procedure that led to \eqref{priories} to \eqref{aode}, we obtain, for every $t\in[0,T_k^\ast)$,
\begin{equation}\label{odepriori}
	\begin{aligned}
		&\|(a_{\Phi_k}^k,\mathbf{u}_{\Phi_k}^k,\sqrt{{\tau}}\mathbf{S}_{\Phi_k}^k)\|_{\dot{B}^{\frac{1}{2},\frac{3}{2}}_{2,1}}
		+\lambda\|(a_{\Phi_k}^k,\mathbf{u}_{\Phi_k}^k,\sqrt{{\tau}}\mathbf{S}_{\Phi_k}^k)\|_{L^1_{t,\dot{\theta}}(\dot{B}^{\frac{3}{2},\frac{5}{2}}_{2,1})}
		+\|(a_{\Phi_k}^k,\mathbf{u}_{\Phi_k}^k,\mathbf{S}_{\Phi_k}^k)\|_{L^1_t(\dot{B}^{\frac{3}{2}}_{2,1})}\\
		&\le C_0\Bigl(\|e^{2\lambda\Lambda}(a_0^k,\mathbf{u}_0^k,\sqrt{{\tau}}\mathbf{S}_0^k)\|_{\dot{B}^{\frac{1}{2},\frac{3}{2}}_{2,1}}
		+\int_0^t \|(a_{\Phi_k}^k,\mathbf{u}_{\Phi_k}^k,\mathbf{S}_{\Phi_k}^k)\|_{\dot{B}^{\frac{3}{2}}_{2,1}}
		\|(a_{\Phi_k}^k,\mathbf{u}_{\Phi_k}^k,\sqrt{{\tau}}\mathbf{S}_{\Phi_k}^k)\|_{\dot{B}^{\frac{3}{2},\frac{5}{2}}_{2,1}}\dif t\Bigr)\\
		&\le C_0\varepsilon_0\lambda_0
		+C_0\int_0^t \dot{\theta}_k(t')\,
		\|(a_{\Phi_k}^k,\mathbf{u}_{\Phi_k}^k,\sqrt{{\tau}}\mathbf{S}_{\Phi_k}^k)(t')\|_{\dot{B}^{\frac{3}{2},\frac{5}{2}}_{2,1}}\dif t'.
	\end{aligned}
\end{equation}
Now, choosing $\lambda=2C_0$,  we get
\begin{equation}\label{odepriori1}
	\begin{aligned}
		&\|(a_{\Phi_k}^k,\mathbf{u}_{\Phi_k}^k,\sqrt{{\tau}}\mathbf{S}_{\Phi_k}^k)\|_{\dot{B}^{\frac{1}{2},\frac{3}{2}}_{2,1}}
		+C_0\|(a_{\Phi_k}^k,\mathbf{u}_{\Phi_k}^k,\sqrt{{\tau}}\mathbf{S}_{\Phi_k}^k)\|_{L^1_{t,\dot{\theta}}(\dot{B}^{\frac{3}{2},\frac{5}{2}}_{2,1})}
		+\|(a_{\Phi_k}^k,\mathbf{u}_{\Phi_k}^k,\mathbf{S}_{\Phi_k}^k)\|_{L^1_t(\dot{B}^{\frac{3}{2}}_{2,1})}\\
		&\le C_0\varepsilon_0\lambda_0.
	\end{aligned}
\end{equation}
Note that, by the choice of \(\varepsilon_0\) above,
\[
C_0\varepsilon_0\le \frac{1}{4C_0}=\frac{1}{2\lambda}.
\]Then \eqref{odepriori1} implies, for all $t\in[0,T_k^\ast)$,
\[
\theta_k(t)\le \frac{\lambda_0}{2\lambda},\qquad
\|(a_{\Phi_k}^k,\mathbf u_{\Phi_k}^k,\sqrt{\tau} \mathbf S_{\Phi_k}^k)\|_{\dot B^{\frac12,\frac32}_{2,1}}\le \frac{M\varepsilon_0}{2}.
\]
By the definition of $T_k^\ast$ and a standard continuity argument, we conclude that $T_k^\ast=T_k$; in fact, the estimates hold uniformly up to $T_k$, and the solution exists globally in time.  

The uniform bounds \eqref{odepriori1} allow us to pass to the limit $k\to\infty$ (after extracting a subsequence) and obtain a solution $(a,\mathbf u,\mathbf S)$ of the original system \eqref{lsys} satisfying the same estimate. The uniqueness follows from the same energy estimates.  We omit the standard compactness details; they are analogous to those in Chapter~10 of \cite{refp1}. Moreover, we can derive for any $t>0$, 
\begin{equation}\label{Apri}
	\begin{aligned}
		\|(a_{\Phi},\mathbf{u}_{\Phi},{\sqrt{\tau}}\mathbf{S}_{\Phi})\|_{\dot{B}^{\frac{1}{2},\frac{3}{2}}_{2,1}}
		+\|(a_{\Phi},\mathbf{u}_{\Phi},\mathbf{S}_{\Phi})\|_{L^1_t(\dot{B}^{\frac{3}{2}}_{2,1})}
		\le C\|e^{2\lambda_0\Lambda}(a_0,\mathbf{u}_0,\sqrt{{\tau}}\mathbf{S}_0)\|_{\dot{B}^{\frac{1}{2},\frac{3}{2}}_{2,1}}.
	\end{aligned}
\end{equation}
By noting that $\lambda_0 |\xi| \le \Phi (t,\xi) \le 2 \lambda_0 |\xi|$, the estimates \eqref{Res} follows immediately.   This completes the proof of Theorem~\ref{th1.1}.

\section{Proof of Theorems \ref{jd1.3} and \ref{th1.2}}
\subsection{Proof of Theorem \ref{jd1.3}}
In this subsection, we give a brief proof of Theorem \ref{jd1.3}, which is similar to that of Theorem \ref{th1.1}.
Indeed, the key step in the proof of Theorem \ref{th1.1} is to establish the a priori estimates for the solutions.
Thus, we only present the a priori estimates below.
To this end, we introduce an analytic weight.
For a function \(f(t,x)\), we define
\begin{equation}\label{definejd}
	f_\Gamma(t,x) := \mathcal{F}^{-1}_{\xi \to x} \bigl( e^{\Gamma(t,\xi)} \hat{f}(t,\xi) \bigr) = e^{\Gamma(t,\Lambda)} f(t,x),
\end{equation}
where \(\Gamma:=(2\lambda_0-\lambda\zeta(t))|\xi|\) and
\[
\zeta(t):=\int_{0}^{t}\left(\|(a_{\Gamma}^0, \mathbf{u}_{\Gamma}^0)\|_{\dot{B}^{3/2}_{2,1}}+\| \mathbf{u}_{\Gamma}^0\|_{\dot{B}^{5/2}_{2,1}}\right)\dif t.
\]
This definition is formulated in exact analogy to the analytic radius defined in \eqref{thetak}, adapted to the limit system by taking the formal limit \(\tau \to 0\) and using the constitutive relation \(\mathbf{S}=2\mu \mathbf{D}(\mathbf{u})\).

Applying the analytic energy method with the weight \(e^{\Gamma(t,\Lambda)}\) to the Navier-Stokes system \eqref{jdxt}, analogous to Proposition \ref{prop1}, we derive the following higher-order a priori estimate for \(s=1,2\)
\begin{align*}
	&	\|(a_{\Gamma}^0,\mathbf{u}_\Gamma^0)\|_{\dot{B}^{\frac{1}{2}+s,\frac{3}{2}+s}_{2,1}}+\lambda\|(a_{\Gamma}^0,\mathbf{u}_\Gamma^0)\|_{L^1_{t,\dot{\zeta}}(\dot{B}^{\frac{3}{2}+s,\frac{5}{2}+s}_{2,1})}+\|(a_{\Gamma}^0,\mathbf{u}_\Gamma^0)\|_{L_t^1(\dot{B}^{\frac{3}{2}+s}_{2,1})}+\|\mathbf{u}_{\Gamma}^0\|_{L^1_t(\dot{B}^{\frac{5}{2}+s}_{2,1})}\\
	&\lesssim\|e^{2\lambda_0\Lambda}(a^0_0,\mathbf{u}^0_0)\|_{\dot{B}^{\frac{1}{2}+s,\frac{3}{2}+s}_{2,1}}+\|(a^0\ddiv\mathbf{u}^0)_{\Gamma}\|_{L^1_t(\dot{B}^{\frac{1}{2}+s,\frac{3}{2}+s}_{2,1})}+\|(\nabla a^0\cdot\mathbf{u}^0)_{\Gamma}\|_{L^1_t(\dot{B}^{\frac{1}{2}+s,\frac{3}{2}+s}_{2,1})}\\
	&\quad+\|(\mathbf{u}^0\cdot\nabla\mathbf{u}^0)_{\Gamma}\|_{L^1_t(\dot{B}^{\frac{1}{2}+s,\frac{3}{2}+s}_{2,1})}+\|([1-\gamma A(1+a^0)^{\gamma -2}]\nabla a^0)_{\Gamma}\|_{L^1_t(\dot{B}^{\frac{1}{2}+s,\frac{3}{2}+s}_{2,1})}\\
	&\quad+\left\|\left(\frac{\mathbf{D}(\mathbf{u}^0)\cdot\nabla a^0}{1+a^0}\right)_{\Gamma}\right\|_{L^1_t(\dot{B}^{\frac{1}{2}+s,\frac{3}{2}+s}_{2,1})}.
\end{align*}

In the following, we shall assume a priori that 	$\|a^0_\Gamma(\cdot, t)\|_{\dot{B}^{\frac{3}{2}}_{2,1}}\le \delta$ for any $t>0$ with sufficiently small $\delta$. This a priori estimates will be closed by the small assumption of the initial data given in \eqref{smalls}. 
For the nonlinear terms, applying Lemmas \ref{fge} and \ref{efa} yields
\begin{align*}
	&	\|(a^0\ddiv\mathbf{u}^0)_{\Gamma}\|_{L^1_t(\dot{B}^{\frac{1}{2}+s,\frac{3}{2}+s}_{2,1})}+\|(\nabla a^0\cdot\mathbf{u}^0)_{\Gamma}\|_{L^1_t(\dot{B}^{\frac{1}{2}+s,\frac{3}{2}+s}_{2,1})}+\|(\mathbf{u}^0\cdot\nabla\mathbf{u}^0)_{\Gamma}\|_{L^1_t(\dot{B}^{\frac{1}{2}+s,\frac{3}{2}+s}_{2,1})} \\
	&\lesssim \int_{0}^{t} 	\Big(\|(a^0_\Gamma,\mathbf{u}^0_{\Gamma})\|_{\dot{B}^{\frac{3}{2}}_{2,1}}\|(\nabla a^0_\Gamma,\nabla\mathbf{u}^0_{\Gamma})\|_{\dot{B}^{\frac{1}{2}+s,\frac{3}{2}+s}_{2,1}}+\|(\nabla a^0_\Gamma,\nabla\mathbf{u}^0_{\Gamma})\|_{\dot{B}^{\frac{3}{2}}_{2,1}}\|(a^0_\Gamma,\mathbf{u}^0_{\Gamma})\|_{\dot{B}^{\frac{1}{2}+s,\frac{3}{2}+s}_{2,1}}\Big)\dif t\\
	&\lesssim \int_{0}^{t} 	\Big(\|(a^0_\Gamma,\mathbf{u}^0_{\Gamma})\|_{\dot{B}^{\frac{3}{2}}_{2,1}}\|( a^0_\Gamma,\mathbf{u}^0_{\Gamma})\|_{\dot{B}^{\frac{3}{2}+s,\frac{5}{2}+s}_{2,1}}+\|( a^0_\Gamma,\mathbf{u}^0_{\Gamma})\|_{\dot{B}^{\frac{5}{2}}_{2,1}}\|(a^0_\Gamma,\mathbf{u}^0_{\Gamma})\|_{\dot{B}^{\frac{1}{2}+s,\frac{3}{2}+s}_{2,1}}\Big)\dif t,
\end{align*}
and
\begin{align*}
	&	\int_{0}^{t} \|([1-\gamma A(1+a^0)^{\gamma -2}]\nabla a^0)_{\Gamma}\|_{\dot{B}^{\frac{1}{2}+s,\frac{3}{2}+s}_{2,1}}	\dif t\\
	&\lesssim \int_{0}^{t} \Big(\|a^0_{\Gamma}\|_{\dot{B}^{\frac{3}{2}}_{2,1}}\| a^0_{\Gamma}\|_{\dot{B}^{\frac{3}{2}+s,\frac{5}{2}+s}_{2,1}}+\sum^{\infty}_{k=1} \|[(a^0)^k]_{\Gamma}\|_{\dot{B}^{\frac{1}{2}+s,\frac{3}{2}+s}_{2,1}}\| a^0_{\Gamma}\|_{\dot{B}^{\frac{5}{2}}_{2,1}}\Big)	\dif t\\
	&\lesssim \int_{0}^{t} \Big(\|a^0_{\Gamma}\|_{\dot{B}^{\frac{3}{2}}_{2,1}}\| a^0_{\Gamma}\|_{\dot{B}^{\frac{3}{2}+s,\frac{5}{2}+s}_{2,1}}+\sum^{\infty}_{k=1} \|a^0_{\Gamma}\|_{\dot{B}^{\frac{1}{2}+s,\frac{3}{2}+s}_{2,1}}\|a^0_{\Gamma}\|^{k-1}_{\dot{B}^{\frac{3}{2}}_{2,1}}\| a^0_{\Gamma}\|_{\dot{B}^{\frac{5}{2}}_{2,1}}\Big)	\dif t\\
	&\lesssim \int_{0}^{t} \Big(\|a^0_{\Gamma}\|_{\dot{B}^{\frac{3}{2}}_{2,1}}\| a^0_{\Gamma}\|_{\dot{B}^{\frac{3}{2}+s,\frac{5}{2}+s}_{2,1}}+ \|a^0_{\Gamma}\|_{\dot{B}^{\frac{1}{2}+s,\frac{3}{2}+s}_{2,1}}\| a^0_{\Gamma}\|_{\dot{B}^{\frac{5}{2}}_{2,1}}\Big)	\dif t.
\end{align*}

On the other hand, observing that
\(\|f\|_{\dot{B}_{2,1}^{s,s^\prime}}\approx\|f\|_{\dot{B}_{2,1}^{s}}+\|f\|_{\dot{B}_{2,1}^{s^\prime}}\)
for \(s<s^\prime\) and arguing similarly as above, we deduce
\begin{align*}
	&	\int_{0}^{t}\left\|\left(\frac{\mathbf{D}(\mathbf{u}^0)\cdot\nabla a^0}{1+a^0}\right)_{\Gamma}\right\|_{\dot{B}^{\frac{1}{2}+s,\frac{3}{2}+s}_{2,1}}	\dif t\\
	&\lesssim 	\int_{0}^{t}\Big(\|\nabla a^0_{\Gamma}\|_{\dot{B}^{\frac{3}{2}}_{2,1}}\|\nabla \mathbf{u}^0_{\Gamma}\|_{\dot{B}^{\frac{3}{2}}_{2,1}}\left\|\left(\frac{1}{1+a^0}\right)_{\Gamma}\right\|_{\dot{B}^{\frac{1}{2}+s,\frac{3}{2}+s}_{2,1}}+\|\nabla a^0_{\Gamma}\|_{\dot{B}^{\frac{1}{2}+s,\frac{3}{2}+s}_{2,1}}\|\nabla \mathbf{u}^0_{\Gamma}\|_{\dot{B}^{\frac{3}{2}}_{2,1}}\left\|\left(\frac{1}{1+a^0}\right)_{\Gamma}\right\|_{\dot{B}^{\frac{3}{2}}_{2,1}} \\
	&\quad \quad+\|\nabla a^0_{\Gamma}\|_{\dot{B}^{\frac{3}{2}}_{2,1}}\|\nabla \mathbf{u}^0_{\Gamma}\|_{\dot{B}^{\frac{1}{2}+s,\frac{3}{2}+s}_{2,1}}\left\|\left(\frac{1}{1+a^0}\right)_{\Gamma}\right\|_{\dot{B}^{\frac{3}{2}}_{2,1}}\Big)	\dif t\\
	&\lesssim 	\int_{0}^{t}\Big(\| a^0_{\Gamma}\|_{\dot{B}^{\frac{5}{2}}_{2,1}}\| \mathbf{u}^0_{\Gamma}\|_{\dot{B}^{\frac{5}{2}}_{2,1}}\sum_{k=0}^{\infty}\|[(a^0)^k]_{\Gamma}\|_{\dot{B}^{\frac{1}{2}+s,\frac{3}{2}+s}_{2,1}}+\| a^0_{\Gamma}\|_{\dot{B}^{\frac{3}{2}+s,\frac{5}{2}+s}_{2,1}}\| \mathbf{u}^0_{\Gamma}\|_{\dot{B}^{\frac{5}{2}}_{2,1}} +\| a^0_{\Gamma}\|_{\dot{B}^{\frac{5}{2}}_{2,1}}\| \mathbf{u}^0_{\Gamma}\|_{\dot{B}^{\frac{3}{2}+s,\frac{5}{2}+s}_{2,1}}\Big)	\dif t\\
	&\lesssim 	\int_{0}^{t}\Big(\| a^0_{\Gamma}\|_{\dot{B}^{\frac{5}{2}}_{2,1}}\| \mathbf{u}^0_{\Gamma}\|_{\dot{B}^{\frac{5}{2}}_{2,1}}(1+\|a^0_{\Gamma}\|_{\dot{B}^{\frac{1}{2}+s,\frac{3}{2}+s}_{2,1}})\sum_{k=0}^{\infty}\|a^0_{\Gamma}\|^k_{\dot{B}^{\frac{3}{2}}_{2,1}}+\| a^0_{\Gamma}\|_{\dot{B}^{\frac{3}{2}+s,\frac{5}{2}+s}_{2,1}}\| \mathbf{u}^0_{\Gamma}\|_{\dot{B}^{\frac{5}{2}}_{2,1}} \\
	&\quad\quad+\| a^0_{\Gamma}\|_{\dot{B}^{\frac{5}{2}}_{2,1}}\| \mathbf{u}^0_{\Gamma}\|_{\dot{B}^{\frac{3}{2}+s,\frac{5}{2}+s}_{2,1}}\Big)	\dif t\\
	&\lesssim 	\int_{0}^{t}\Big(\| a^0_{\Gamma}\|_{\dot{B}^{\frac{5}{2}}_{2,1}}\| \mathbf{u}^0_{\Gamma}\|_{\dot{B}^{\frac{5}{2}}_{2,1}}(1+\|a^0_{\Gamma}\|_{\dot{B}^{\frac{1}{2}+s,\frac{3}{2}+s}_{2,1}})+\| a^0_{\Gamma}\|_{\dot{B}^{\frac{3}{2}+s,\frac{5}{2}+s}_{2,1}}\| \mathbf{u}^0_{\Gamma}\|_{\dot{B}^{\frac{5}{2}}_{2,1}} +\| a^0_{\Gamma}\|_{\dot{B}^{\frac{5}{2}}_{2,1}}(\| \mathbf{u}^0_{\Gamma}\|_{\dot{B}^{\frac{3}{2}+s}_{2,1}}+\| \mathbf{u}^0_{\Gamma}\|_{\dot{B}^{\frac{5}{2}+s}_{2,1}})\Big)	\dif t\\
	&\lesssim 	\int_{0}^{t}\Big(\| \mathbf{u}^0_{\Gamma}\|_{\dot{B}^{\frac{5}{2}}_{2,1}}\|a^0_{\Gamma}\|_{\dot{B}^{\frac{1}{2}+s,\frac{3}{2}+s}_{2,1}}+\|\mathbf{u}^0_{\Gamma}\|_{\dot{B}^{\frac{5}{2}}_{2,1}}\|a^0_{\Gamma}\|_{\dot{B}^{\frac{3}{2},\frac{5}{2}}_{2,1}}\|a^0_{\Gamma}\|_{\dot{B}^{\frac{1}{2}+s,\frac{3}{2}+s}_{2,1}}+\| a^0_{\Gamma}\|_{\dot{B}^{\frac{3}{2}+s,\frac{5}{2}+s}_{2,1}}\| \mathbf{u}^0_{\Gamma}\|_{\dot{B}^{\frac{5}{2}}_{2,1}} \\
	&\quad\quad+\| a^0_{\Gamma}\|_{\dot{B}^{\frac{5}{2}}_{2,1}}\| \mathbf{u}^0_{\Gamma}\|_{\dot{B}^{\frac{1}{2}+s,\frac{3}{2}+s}_{2,1}}+\|a^0_{\Gamma}\|_{\dot{B}^{\frac{3}{2},\frac{5}{2}}_{2,1}}\| \mathbf{u}^0_{\Gamma}\|_{\dot{B}^{\frac{5}{2}+s}_{2,1}}\Big)	\dif t.
\end{align*}

Combining the above estimates, we obtain
\begin{equation}\label{jdgjzex}
	\begin{aligned}
		&	\|(a_{\Gamma}^0,\mathbf{u}_\Gamma^0)\|_{\dot{B}^{\frac{1}{2}+s,\frac{3}{2}+s}_{2,1}}+\lambda\|(a_{\Gamma}^0,\mathbf{u}_\Gamma^0)\|_{L^1_{t,\dot{\zeta}}(\dot{B}^{\frac{3}{2}+s,\frac{5}{2}+s}_{2,1})}+\|(a_{\Gamma}^0,\mathbf{u}_\Gamma^0)\|_{L_t^1(\dot{B}^{\frac{3}{2}+s}_{2,1})}+\|\mathbf{u}_{\Gamma}^0\|_{L^1_t(\dot{B}^{\frac{5}{2}+s}_{2,1})}\\
		&\lesssim\|e^{2\lambda_0\Lambda}(a^0_0,\mathbf{u}^0_0)\|_{\dot{B}^{\frac{1}{2}+s,\frac{3}{2}+s}_{2,1}}+\|(a_{\Gamma}^0,\mathbf{u}_\Gamma^0)\|_{L^1_{t,\dot{\zeta}}(\dot{B}^{\frac{3}{2}+s,\frac{5}{2}+s}_{2,1})}+\int_{0}^{t} 	\|( a^0_\Gamma,\mathbf{u}^0_{\Gamma})\|_{\dot{B}^{\frac{5}{2}}_{2,1}}\|(a^0_\Gamma,\mathbf{u}^0_{\Gamma})\|_{\dot{B}^{\frac{1}{2}+s,\frac{3}{2}+s}_{2,1}}\dif t\\
		&\quad+\int_{0}^{t}\Big(\|\mathbf{u}^0_{\Gamma}\|_{\dot{B}^{\frac{5}{2}}_{2,1}}\|a^0_{\Gamma}\|_{\dot{B}^{\frac{3}{2},\frac{5}{2}}_{2,1}}\|a^0_{\Gamma}\|_{\dot{B}^{\frac{1}{2}+s,\frac{3}{2}+s}_{2,1}}+\|a^0_{\Gamma}\|_{\dot{B}^{\frac{3}{2},\frac{5}{2}}_{2,1}}\| \mathbf{u}^0_{\Gamma}\|_{\dot{B}^{\frac{5}{2}+s}_{2,1}}\Big)	\dif t.
	\end{aligned}
\end{equation}
Then, by choosing \(\lambda\) sufficiently large and taking \(\varepsilon_0\) sufficiently small in \eqref{smalls}, we can absorb the nonlinear terms and conclude that
\begin{equation}\label{jdgj1}
	\begin{aligned}
		\|(a_{\Gamma}^0,\mathbf{u}_\Gamma^0)\|_{L^\infty(\dot{B}^{\frac{1}{2}+s,\frac{3}{2}+s}_{2,1})}+\|(a_{\Gamma}^0,\mathbf{u}_\Gamma^0)\|_{L^1(\dot{B}^{\frac{3}{2}+s}_{2,1})}+\|\mathbf{u}_{\Gamma}^0\|_{L^1(\dot{B}^{\frac{5}{2}+s}_{2,1})}
		\lesssim\|e^{2\lambda_0\Lambda}(a^0_0,\mathbf{u}^0_0)\|_{\dot{B}^{\frac{1}{2}+s,\frac{3}{2}+s}_{2,1}},
	\end{aligned}
\end{equation}
where \(s=1,2\). The global existence and uniqueness of the solution follow from a standard approximation and continuity argument, entirely analogous to the proof of Theorem \ref{th1.1}. This completes the proof of Theorem \ref{jd1.3}.

\subsection{Proof of Theorem \ref{th1.2}}
Now, we consider the error equations for the classical Navier-Stokes equation \eqref{jdxt} and the viscoelastic model \eqref{ysys}. To do this, we first rewrite  the limiting system \eqref{jdxt} in the following form:
\begin{equation}\label{jdns}
	\begin{cases}
		\rho_t+\nabla\cdot(\rho \mathbf{u})=0,\\
		\rho\mathbf{u}_t+\rho\mathbf{u}\cdot\nabla\mathbf{u}+\nabla p=\ddiv (\rho \mathbf{S}),\\
		\mathbf{S}=2\mu\mathbf{D}(\mathbf{u}).
	\end{cases}
\end{equation}
Let \((\rho^\tau, \mathbf{u}^\tau, \mathbf{S}^\tau)\) denote the solutions of the Cauchy problem for the UCM model \eqref{ysys}, as given by Theorem \ref{th1.1}, and let \((\rho^0, \mathbf{u}^0)\) be the solution of the classical Navier-Stokes equations \eqref{jdxt}--\eqref{jdxt1.1} obtained in Theorem \ref{jd1.3}. Define the differences
\[
(\tilde a, \tilde{\mathbf u}, \tilde{\mathbf S}):=(\rho^\tau-\rho^0,\; \mathbf{u}^\tau-\mathbf{u}^0,\; \mathbf{S}^\tau-\mathbf{S}^0),
\]
where $\mathbf{S}^0=2 \mu \mathbf{D}(\mathbf{u^0})$. 
Then, we derive the error equations as follows
\begin{equation}\label{rdfc}
	\begin{cases}
		\tilde a_t+\ddiv \tilde{\mathbf u}=\tilde F+\bar F,\\[1mm]
		\tilde{\mathbf u}_t+\nabla\tilde a-\ddiv\tilde{\mathbf S}=\tilde G+\bar G,\\[1mm]
		\tau\tilde{\mathbf S}_t+\tilde{\mathbf S}-2\mu\mathbf{D}(\tilde{\mathbf u})
		=\tau \tilde H+\tau \bar H,
	\end{cases}
\end{equation}
where the nonlinear terms are given by
\begin{align*}
	\tilde F &= -\nabla \tilde a\cdot\mathbf{u}^\tau-a^\tau\ddiv\tilde{\mathbf u},\\
	\bar F &= -\nabla a^0\cdot\tilde{\mathbf u}-\tilde a\ddiv\mathbf{u}^0,\\
	\tilde G &= -\mathbf{u}^\tau\cdot\nabla\tilde{\mathbf u}+[1-(a^\tau+1)^{\gamma-2}]\nabla\tilde a
	+\mathbf{S}^\tau\cdot\frac{\nabla \tilde a}{a^\tau+1},\\
	\bar G &= -\tilde{\mathbf u}\cdot\nabla\mathbf{u}^0
	+\gamma A\bigl[(a^0+1)^{\gamma-2}-(a^\tau+1)^{\gamma-2}\bigr]\nabla a^0\\
	&\quad+\Bigl[\frac{\tilde{\mathbf S}}{a^\tau+1}+\mathbf{S}^0\frac{\tilde a}{(a^\tau+1)(a^0+1)}\Bigr] \cdot\nabla a^0,\\
	\tilde H &= -\mathbf{u}^\tau\cdot\nabla\tilde{\mathbf S}
	+\nabla\tilde{\mathbf u}\cdot\mathbf{S}^\tau+\mathbf{S}^\tau\cdot(\nabla\tilde{\mathbf u})^T,\\
	\bar H &= -\mathbf{u}^\tau\cdot\nabla\mathbf{S}^0
	+\nabla\mathbf{u}^0\cdot\mathbf{S}^\tau+\mathbf{S}^\tau\cdot(\nabla{\mathbf{u}^0})^T+\mathbf{S}^0_t.
\end{align*}

Now set
\[
\eta(t):=\int_{0}^{t} \bigl(\|(a^\tau_{\Phi}, \mathbf{u}^\tau_{\Phi}, \mathbf{S}^\tau_{\Phi})\|_{\dot{B}^{\frac{3}{2}}_{2,1}}
+\|(a_{\Gamma}^0, \mathbf{u}_{\Gamma}^0)\|_{\dot{B}^{\frac{3}{2}}_{2,1}}+\|\mathbf{u}^0_\Gamma\|_{\dot B^\frac{5}{2}_{2,1}}\bigr)\,\dif s,
\]
and define
\begin{equation}\label{definerd}
	f_\Psi(t,x) := \mathcal{F}^{-1}_{\xi \to x} \bigl( e^{\Psi(t,\xi)} \hat{f}(t,\xi) \bigr) = e^{\Psi(t,\Lambda)} f(t,x),
\end{equation}
with $\Psi(t,\xi):=(2\lambda_0-\nu\eta(t))|\xi|$ and $\nu\ge\lambda$ to be chosen later. For sufficiently small $\varepsilon_0$ given by Theorem \ref{th1.1} and Corollary \ref{jd1.1}, we have
\[
\Psi(t,\xi)\le \min\{\Phi(t,\xi),\Gamma(t,\xi)\},\qquad \Psi(t,\xi)\ge 0.
\]

Applying the transform $e^{\Psi(t,\Lambda)}$ to the error system \eqref{rdfc} and projecting onto dyadic blocks, we obtain for each $k$:
\begin{equation}\label{rdjdfc}
	\begin{cases}
		\partial_t\tilde a_{\Psi,k}+\nu\dot{\eta}(t)\Lambda\tilde a_{\Psi,k}+\ddiv \tilde{\mathbf u}_{\Psi,k}
		=\tilde F_{\Psi,k}+\bar F_{\Psi,k},\\[1mm]
		\partial_t\tilde{\mathbf u}_{\Psi,k}+\nu\dot{\eta}(t)\Lambda\tilde{\mathbf u}_{\Psi,k}
		+\nabla\tilde a_{\Psi,k}-\ddiv\tilde{\mathbf S}_{\Psi,k}
		=\tilde G_{\Psi,k}+\bar G_{\Psi,k},\\[1mm]
		\tau\partial_t\tilde{\mathbf S}_{\Psi,k}+\tau\nu\dot{\eta}(t)\Lambda\tilde{\mathbf S}_{\Psi,k}
		+\tilde{\mathbf S}_{\Psi,k}-2\mu\mathbf{D}(\tilde{\mathbf u})_{\Psi,k}
		=\tau \tilde H_{\Psi,k}+\tau \bar H_{\Psi,k}.
	\end{cases}
\end{equation}

Multiplying the three equations by $2\tilde a_{\Psi,k}$, $2\tilde{\mathbf u}_{\Psi,k}$,  and $\frac1\mu\tilde{\mathbf S}_{\Psi,k}$, respectively, integrating over $\mathbb R^3$, and summing the results, we get
\begin{equation}\label{4.8}
\begin{aligned}
	\frac12\frac{\dif}{\dif t}
	&\left(2\|\tilde a_{\Psi,k}\|_{L^2}^2+2\|\tilde{\mathbf u}_{\Psi,k}\|_{L^2}^2+\frac{\tau}{\mu}\|\tilde{\mathbf S}_{\Psi,k}\|_{L^2}^2\right)
	\\
	&+\nu\dot{\eta}(t)\left(2\|\Lambda^{\frac12}\tilde a_{\Psi,k}\|_{L^2}^2
	+2\|\Lambda^{\frac12}\tilde{\mathbf u}_{\Psi,k}\|_{L^2}^2
	+\frac{\tau}{\mu}\|\Lambda^{\frac12}\tilde{\mathbf S}_{\Psi,k}\|_{L^2}^2\right)
	+\frac1\mu\|\tilde{\mathbf S}_{\Psi,k}\|_{L^2}^2\\
	&=2\langle \tilde F_{\Psi,k},\tilde a_{\Psi,k}\rangle+2\langle \bar F_{\Psi,k},\tilde a_{\Psi,k}\rangle
	+2\langle \tilde G_{\Psi,k},\tilde{\mathbf u}_{\Psi,k}\rangle+2\langle \bar G_{\Psi,k},\tilde{\mathbf u}_{\Psi,k}\rangle\\
	&\quad+\frac{\tau}{\mu}\langle \tilde H_{\Psi,k},\tilde{\mathbf S}_{\Psi,k}\rangle
	+\frac{\tau}{\mu}\langle \bar H_{\Psi,k},\tilde{\mathbf S}_{\Psi,k}\rangle .
\end{aligned}
\end{equation}
\medskip
\noindent\textbf{High-frequency estimates.}
Applying $\Lambda^{-1}\nabla$ to the first equation in \eqref{rdjdfc}, $\Lambda^{-1}$ to the second, and $\Lambda^{-1}\ddiv$ to the third, we obtain
\begin{equation}\label{rdhfg}
	\begin{cases}
		\partial_t\Lambda^{-1}\nabla\tilde a_{\Psi,k}
		+\nu\dot\eta(t)\nabla\tilde a_{\Psi,k}
		+\Lambda^{-1}\nabla\mathrm{div}\,\tilde{\mathbf u}_{\Psi,k}
		=\Lambda^{-1}\nabla(\tilde F_{\Psi,k}+\bar F_{\Psi,k}),\\[1mm]
		\partial_t\Lambda^{-1}\tilde{\mathbf u}_{\Psi,k}
		+\nu\dot\eta(t)\tilde{\mathbf u}_{\Psi,k}
		+\Lambda^{-1}\nabla\tilde a_{\Psi,k}
		-\Lambda^{-1}\mathrm{div}\,\tilde{\mathbf S}_{\Psi,k}
		=\Lambda^{-1}(\tilde G_{\Psi,k}+\bar G_{\Psi,k}),\\[1mm]
		\tau\partial_t\Lambda^{-1}\mathrm{div}\,\tilde{\mathbf S}_{\Psi,k}
		+\tau\nu\dot\eta(t)\mathrm{div}\,\tilde{\mathbf S}_{\Psi,k}
		+\Lambda^{-1}\mathrm{div}\,\tilde{\mathbf S}_{\Psi,k}
		+\mu(\Lambda-\nabla\Lambda^{-1}\mathrm{div})\tilde{\mathbf u}_{\Psi,k}\\
		\qquad\qquad\qquad\qquad\qquad\qquad\qquad\qquad\qquad
		=\Lambda^{-1}\mathrm{div}(\tau\tilde H_{\Psi,k}+\tau\bar H_{\Psi,k}).
	\end{cases}
\end{equation}

Taking the $L^2$ inner product of the first equation with $\Lambda^{-1}\tilde{\mathbf u}_{\Psi,k}$, and the second with $\Lambda^{-1}\nabla\tilde a_{\Psi,k}$, then adding the results, yields
\begin{equation}\label{rdhpgj}
	\begin{aligned}
		&\partial_t\langle \Lambda^{-1}\nabla\tilde a_{\Psi,k},\Lambda^{-1}\tilde{\mathbf u}_{\Psi,k}\rangle
		+2\nu\dot\eta(t)\langle \nabla\tilde a_{\Psi,k},\Lambda^{-1}\tilde{\mathbf u}_{\Psi,k}\rangle
		-\|\Lambda^{-1}\mathrm{div}\,\tilde{\mathbf u}_{\Psi,k}\|_{L^2}^2+\|\tilde a_{\Psi,k}\|_{L^2}^2 \\
		&\quad-\langle \Lambda^{-1}\mathrm{div}\,\tilde{\mathbf S}_{\Psi,k},\Lambda^{-1}\nabla\tilde a_{\Psi,k}\rangle \\
		&= \langle \Lambda^{-1}\nabla(\tilde F_{\Psi,k}+\bar F_{\Psi,k}),\Lambda^{-1}\tilde{\mathbf u}_{\Psi,k}\rangle
		+\langle \Lambda^{-1}(\tilde G_{\Psi,k}+\bar G_{\Psi,k}),\Lambda^{-1}\nabla\tilde a_{\Psi,k}\rangle .
	\end{aligned}
\end{equation}

Next, we take the inner product of the second equation with
\(\frac{\tau}{\mu}\Lambda^{-1}\mathrm{div}\,\tilde{\mathbf S}_{\Psi,k}\)
and of the third equation with
\(\frac{1}{\mu}\Lambda^{-1}\tilde{\mathbf u}_{\Psi,k}\).
Adding the two resulting identities gives
\begin{equation}\label{rddpgj}
	\begin{aligned}
		&\frac{\tau}{\mu}\partial_t\langle \Lambda^{-1}\tilde{\mathbf u}_{\Psi,k},
		\Lambda^{-1}\mathrm{div}\,\tilde{\mathbf S}_{\Psi,k}\rangle
		+\frac{2\tau}{\mu}\nu\dot\eta(t)
		\langle \Lambda^{-1}\tilde{\mathbf u}_{\Psi,k},\mathrm{div}\,\tilde{\mathbf S}_{\Psi,k}\rangle \\
		&\quad+\frac{\tau}{\mu}\langle \Lambda^{-1}\nabla\tilde a_{\Psi,k},
		\Lambda^{-1}\mathrm{div}\,\tilde{\mathbf S}_{\Psi,k}\rangle
		-\frac{\tau}{\mu}\|\Lambda^{-1}\mathrm{div}\,\tilde{\mathbf S}_{\Psi,k}\|_{L^2}^2 \\
		&\quad+\frac1\mu\langle \Lambda^{-1}\mathrm{div}\,\tilde{\mathbf S}_{\Psi,k},
		\Lambda^{-1}\tilde{\mathbf u}_{\Psi,k}\rangle
		+\|\tilde{\mathbf u}_{\Psi,k}\|_{L^2}^2
		+\|\Lambda^{-1}\mathrm{div}\,\tilde{\mathbf u}_{\Psi,k}\|_{L^2}^2 \\
		&= \frac{\tau}{\mu}\langle \Lambda^{-1}(\tilde G_{\Psi,k}+\bar G_{\Psi,k}),
		\Lambda^{-1}\mathrm{div}\,\tilde{\mathbf S}_{\Psi,k}\rangle
		+\frac1\mu\langle \Lambda^{-1}\mathrm{div}(\tau\tilde H_{\Psi,k}+\tau\bar H_{\Psi,k}),
		\Lambda^{-1}\tilde{\mathbf u}_{\Psi,k}\rangle .
	\end{aligned}
\end{equation}

For $k\ge 3$, define
\[
\bar {\mathcal{E}}_k^2:=
2\|\tilde a_{\Psi,k}\|_{L^2}^2+2\|\tilde{\mathbf u}_{\Psi,k}\|_{L^2}^2
+\frac{\tau}{\mu}\|\tilde{\mathbf S}_{\Psi,k}\|_{L^2}^2
+\langle \Lambda^{-1}\nabla\tilde a_{\Psi,k},\Lambda^{-1}\tilde{\mathbf u}_{\Psi,k}\rangle
+\langle \Lambda^{-1}\tilde{\mathbf u}_{\Psi,k},\frac{\tau}{\mu}\Lambda^{-1}\mathrm{div}\,\tilde{\mathbf S}_{\Psi,k}\rangle,
\]
\[
\mathcal{E}_k^2:=\|\tilde a_{\Psi,k}\|_{L^2}^2+\|\tilde{\mathbf u}_{\Psi,k}\|_{L^2}^2+\frac{\tau}{\mu}\|\tilde{\mathbf S}_{\Psi,k}\|_{L^2}^2,
\qquad
\mathcal{D}_k^2:=\|\tilde a_{\Psi,k}\|_{L^2}^2+\|\tilde{\mathbf u}_{\Psi,k}\|_{L^2}^2+\|\tilde{\mathbf S}_{\Psi,k}\|_{L^2}^2.
\]
By Bernstein's inequality, we have
\[
\frac12 \mathcal{E}_k^2\le \bar {\mathcal E}_k^2\le 3 \mathcal{E}_k^2.
\]
Adding the energy identities for \eqref{4.8}, \eqref{rdhpgj} and \eqref{rddpgj}, and using H\"older's inequality together with the above equivalence, we obtain
\[
\frac{\dif}{\dif t}\bar  {\mathcal E}_k^2
+\frac14\nu\dot\eta(t)2^k\bar  {\mathcal E}_k^2
+\frac18 \mathcal D_k^2
\lesssim\Bigl(
\|(\tilde F_{\Psi,k},\tilde G_{\Psi,k}, \sqrt{\frac{\tau}{\mu}}\tilde H_{\Psi,k})\|_{L^2}
+\|(\bar F_{\Psi,k},\bar G_{\Psi,k}, \sqrt{\frac{\tau}{\mu}}\bar H_{\Psi,k})\|_{L^2}
\Bigr)\bar  {\mathcal E}_k .
\]
Since $\mathcal E_k\le  \mathcal D_k$, we get
\begin{equation}\label{rdgpbe}
	\frac{\dif}{\dif t} \bar  {\mathcal E}_k
	+\frac14\nu\dot\eta(t)2^k \bar  {\mathcal E}_k
	+\frac18 \mathcal{D}_k
	\lesssim
	\|(\tilde F_{\Psi,k},\tilde G_{\Psi,k}, \sqrt{\frac{\tau}{\mu}}\tilde H_{\Psi,k})\|_{L^2}
	+\|(\bar F_{\Psi,k},\bar G_{\Psi,k}, \sqrt{\frac{\tau}{\mu}}\bar H_{\Psi,k})\|_{L^2}
	.
\end{equation}
Multiplying \eqref{rdgpbe} by $2^{\frac {3}{2} k}$, integrating in time and summing over $k\ge3$, we obtain the high-frequency estimate
\begin{equation}\label{rdgpbesov}
	\begin{aligned}
		&\|(\tilde a_{\Psi},\tilde{\mathbf u}_{\Psi},\sqrt{\frac{\tau}{\mu}}\,\tilde{\mathbf S}_{\Psi})\|^h_{\dot B_{2,1}^{\frac{3}{2}}}
		+\nu\int_0^t\dot\eta(t)
		\|(\tilde a_{\Psi},\tilde{\mathbf u}_{\Psi},\sqrt{\frac{\tau}{\mu}}\,\tilde{\mathbf S}_{\Psi})\|^h_{\dot B_{2,1}^{\frac{5}{2}}}\,\dif t \\
		&\quad+\int_0^t\|(\tilde a_{\Psi},\tilde{\mathbf u}_{\Psi},\tilde{\mathbf S}_{\Psi})\|^h_{\dot B_{2,1}^{\frac{3}{2}}}\,\dif t \\
		&\lesssim
		\|(\tilde a_{\Psi},\tilde{\mathbf u}_{\Psi},\sqrt{\frac{\tau}{\mu}}\,\tilde{\mathbf S}_{\Psi})(0)\|^h_{\dot B_{2,1}^{\frac{3}{2}}}
		+\int_0^t
		\bigl(\|(\tilde F_{\Psi},\tilde G_{\Psi},\sqrt{\frac{\tau}{\mu}}\tilde H_{\Psi})\|^h_{\dot B_{2,1}^{\frac{3}{2}}}
		+\|(\bar F_{\Psi},\bar G_{\Psi}, \sqrt{\frac{\tau}{\mu}}\bar H_{\Psi})\|^h_{\dot B_{2,1}^{\frac{3}{2}}}
		\bigr)\,\dif t
		.
	\end{aligned}
\end{equation}

\medskip
\noindent\textbf{Low-frequency estimates.}
Applying $\nabla$ to the first equation in \eqref{rdjdfc} and $\ddiv$ to the third gives
\begin{equation}\label{rdlfg}
	\begin{cases}
		\partial_t\nabla\tilde a_{\Psi,k}+\nu\dot\eta(t)\nabla\Lambda\tilde a_{\Psi,k}
		+\nabla\mathrm{div}\,\tilde{\mathbf u}_{\Psi,k}
		=\nabla(\tilde F_{\Psi,k}+\bar F_{\Psi,k}),\\[1mm]
		\partial_t\tilde{\mathbf u}_{\Psi,k}+\nu\dot\eta(t)\Lambda\tilde{\mathbf u}_{\Psi,k}
		+\nabla\tilde a_{\Psi,k}-\mathrm{div}\,\tilde{\mathbf S}_{\Psi,k}
		=\tilde G_{\Psi,k}+\bar G_{\Psi,k},\\[1mm]
		\tau\partial_t\mathrm{div}\,\tilde{\mathbf S}_{\Psi,k}
		+\tau\nu\dot\eta(t)\mathrm{div}\Lambda\tilde{\mathbf S}_{\Psi,k}
		+\mathrm{div}\,\tilde{\mathbf S}_{\Psi,k}
		-\mu(\Delta+\nabla\mathrm{div})\tilde{\mathbf u}_{\Psi,k}
		=\mathrm{div}(\tau\tilde H_{\Psi,k}+\tau\bar H_{\Psi,k}).
	\end{cases}
\end{equation}

Taking the inner product of the first equation with
\(\tilde{\mathbf u}_{\Psi,k}\)
and of the second equation with
\(\nabla\tilde a_{\Psi,k}\),
and adding the resulting identities, we obtain\begin{equation}\label{3.11}
	\begin{aligned}
		&\partial_t\langle \nabla\tilde a_{\Psi,k},\tilde{\mathbf u}_{\Psi,k}\rangle
		+2\nu\dot\eta(t)\langle \nabla\Lambda\tilde a_{\Psi,k},\tilde{\mathbf u}_{\Psi,k}\rangle
		-\|\mathrm{div}\,\tilde{\mathbf u}_{\Psi,k}\|_{L^2}^2
		+\|\nabla\tilde a_{\Psi,k}\|_{L^2}^2
		-\langle \mathrm{div}\,\tilde{\mathbf S}_{\Psi,k},\nabla\tilde a_{\Psi,k}\rangle \\
		&= \langle \nabla(\tilde F_{\Psi,k}+\bar F_{\Psi,k}),\tilde{\mathbf u}_{\Psi,k}\rangle
		+\langle \tilde G_{\Psi,k}+\bar G_{\Psi,k},\nabla\tilde a_{\Psi,k}\rangle .
	\end{aligned}
\end{equation}

Next, take the inner product of the second equation with $\frac{\tau}{\mu}\mathrm{div}\,\tilde{\mathbf S}_{\Psi,k}$, and the third with $\frac1\mu\tilde{\mathbf u}_{\Psi,k}$, add and integrate by parts, one has
\begin{equation}\label{3.12}
	\begin{aligned}
		&\frac{\tau}{\mu}\partial_t\langle \tilde{\mathbf u}_{\Psi,k},
		\mathrm{div}\,\tilde{\mathbf S}_{\Psi,k}\rangle
		+\frac{2\tau}{\mu}\nu\dot\eta(t)
		\langle \Lambda\tilde{\mathbf u}_{\Psi,k},\mathrm{div}\,\tilde{\mathbf S}_{\Psi,k}\rangle \\
		&\quad+\frac{\tau}{\mu}\langle \nabla\tilde a_{\Psi,k},
		\mathrm{div}\,\tilde{\mathbf S}_{\Psi,k}\rangle
		-\frac{\tau}{\mu}\|\mathrm{div}\,\tilde{\mathbf S}_{\Psi,k}\|_{L^2}^2
		+\frac1\mu\langle \mathrm{div}\,\tilde{\mathbf S}_{\Psi,k},\tilde{\mathbf u}_{\Psi,k}\rangle
		+\|\Lambda\tilde{\mathbf u}_{\Psi,k}\|_{L^2}^2
		+\|\mathrm{div}\,\tilde{\mathbf u}_{\Psi,k}\|_{L^2}^2 \\
		&= \frac{\tau}{\mu}\langle \tilde G_{\Psi,k}+\bar G_{\Psi,k},
		\mathrm{div}\,\tilde{\mathbf S}_{\Psi,k}\rangle
		+\frac1\mu\langle \mathrm{div}(\tau\tilde H_{\Psi,k}+\tau\bar H_{\Psi,k}),\tilde{\mathbf u}_{\Psi,k}\rangle .
	\end{aligned}
\end{equation}

For $k\le 3$, define
\[
\bar {\mathcal E}_k^2:=
2\|\tilde a_{\Psi,k}\|_{L^2}^2+2\|\tilde{\mathbf u}_{\Psi,k}\|_{L^2}^2
+\frac{\tau}{\mu}\|\tilde{\mathbf S}_{\Psi,k}\|_{L^2}^2
+\frac{3}{8}2^{-k}\langle \nabla\tilde a_{\Psi,k},\tilde{\mathbf u}_{\Psi,k}\rangle
+\frac{3}{8}2^{-k}\langle \tilde{\mathbf u}_{\Psi,k},\frac{\tau}{\mu}\mathrm{div}\,\tilde{\mathbf S}_{\Psi,k}\rangle,
\]
\[
\bar {\mathcal D}_k^2:=\|\tilde a_{\Psi,k}\|_{L^2}^2+\|\tilde{\mathbf u}_{\Psi,k}\|_{L^2}^2+2^{-2k}\|\tilde{\mathbf S}_{\Psi,k}\|_{L^2}^2,
\]
and let $\mathcal E_k, \mathcal D_k$ be as before. Bernstein's inequality gives, 
\[
\frac12  \mathcal E_k^2\le \bar {\mathcal E}_k^2\le 3 \mathcal{E}_k^2.
\]
Adding the identities from \eqref{4.8}, $\frac38 2^{-k}\times$\eqref{3.11}, and $\frac38 2^{-k}\times$\eqref{3.12}, then using H\"older's inequality, we obtain
\[
\frac{\dif}{\dif t}\bar {\mathcal E}_k^2
+\frac14\nu\dot\eta(t)2^k\bar {\mathcal E}_k^2
+\frac1{2\sqrt2}2^k \mathcal D_k\bar {\mathcal D}_k
\lesssim\Bigl(
\|(\tilde F_{\Psi,k},\tilde G_{\Psi,k}, \sqrt{\frac{\tau}{\mu}}\tilde H_{\Psi,k})\|_{L^2}
+\|(\bar F_{\Psi,k},\bar G_{\Psi,k}, \sqrt{\frac{\tau}{\mu}}\bar H_{\Psi,k})\|_{L^2}
\Bigr)\bar {\mathcal E}_k,
\]
and hence
\begin{equation}\label{3.13}
	\frac{\dif}{\dif t}\bar {\mathcal E}_k
	+\frac14\nu\dot\eta(t)2^k \bar {\mathcal E}_k
	+\frac1{2\sqrt2}2^k \bar{\mathcal D}_k
	\lesssim
	\|(\tilde F_{\Psi,k},\tilde G_{\Psi,k}, \sqrt{\frac{\tau}{\mu}}\tilde H_{\Psi,k})\|_{L^2}
	+\|(\bar F_{\Psi,k},\bar G_{\Psi,k}, \sqrt{\frac{\tau}{\mu}}\bar H_{\Psi,k})\|_{L^2}
	.
\end{equation}
Multiplying \eqref{3.13} by $2^{\frac 3 2 k}$, integrating over $(0,t)$ and summing over $k\le 3$, we get the low-frequency estimate
\begin{equation}\label{3.14}
	\begin{aligned}
		&\|(\tilde a_{\Psi},\tilde{\mathbf u}_{\Psi},\sqrt{\frac{\tau}{\mu}}\,\tilde{\mathbf S}_{\Psi})\|^l_{\dot B_{2,1}^{\frac{3}{2}}}
		+\nu\int_0^t\dot\eta(t)
		\|(\tilde a_{\Psi},\tilde{\mathbf u}_{\Psi},\sqrt{\frac{\tau}{\mu}}\,\tilde{\mathbf S}_{\Psi})\|^l_{\dot B_{2,1}^{\frac{5}{2}}}\,\dif t \\
		&\quad+\int_0^t\|(\tilde a_{\Psi},\tilde{\mathbf u}_{\Psi})\|^l_{\dot B_{2,1}^{\frac{5}{2}}}\,\dif t
		+\int_0^t\|\tilde{\mathbf S}_{\Psi}\|^l_{\dot B_{2,1}^{\frac{3}{2}}}\,\dif t \\
		&\lesssim
		\|(\tilde a_{\Psi},\tilde{\mathbf u}_{\Psi},\sqrt{\frac{\tau}{\mu}}\,\tilde{\mathbf S}_{\Psi})(0)\|^l_{\dot B_{2,1}^{\frac{3}{2}}}
		+\int_0^t
		\bigl(\|(\tilde F_{\Psi},\tilde G_{\Psi}, \sqrt{\frac{\tau}{\mu}}\tilde H_{\Psi})\|^l_{\dot B_{2,1}^{\frac{3}{2}}}
		+\|(\bar F_{\Psi},\bar G_{\Psi}, \sqrt{\frac{\tau}{\mu}}\bar H_{\Psi})\|^l_{\dot B_{2,1}^{\frac{3}{2}}}
		\bigr)\,\dif t.
	\end{aligned}
\end{equation}

Combining the high-frequency estimate \eqref{rdgpbesov} and the low-frequency estimate \eqref{3.14}, we obtain 
\begin{equation}\label{rdxygj}
	\begin{aligned}
		&\|(\tilde a_{\Psi},\tilde{\mathbf u}_{\Psi},\sqrt{\frac{\tau}{\mu}}\,\tilde{\mathbf S}_{\Psi})\|_{\dot B_{2,1}^{\frac{3}{2}}}
		+\nu\int_0^t\dot\eta(t)
		\|(\tilde a_{\Psi},\tilde{\mathbf u}_{\Psi},\sqrt{\frac{\tau}{\mu}}\,\tilde{\mathbf S}_{\Psi})\|_{\dot B_{2,1}^{\frac{5}{2}}}\,\dif t \\
		&\quad+\int_0^t\|(\tilde a_{\Psi},\tilde{\mathbf u}_{\Psi})\|_{\dot B_{2,1}^{\frac{5}{2},\frac{3}{2}}}\,\dif t
		+\int_0^t\|\tilde{\mathbf S}_{\Psi}\|_{\dot B_{2,1}^{\frac{3}{2}}}\,\dif t \\
		&\lesssim
		\|e^{2\lambda\Lambda}(\tilde a_0,\tilde{\mathbf u}_0,\sqrt{\frac{\tau}{\mu}}\,\tilde{\mathbf S}_0)\|_{\dot B_{2,1}^{\frac{3}{2}}}
		+\int_0^t
		\bigl(\|(\tilde F_{\Psi},\tilde G_{\Psi}, \sqrt{\frac{\tau}{\mu}}\tilde H_{\Psi})\|_{\dot B_{2,1}^{\frac{3}{2}}}
		+\|(\bar F_{\Psi},\bar G_{\Psi}, \sqrt{\frac{\tau}{\mu}}\bar H_{\Psi})\|_{\dot B_{2,1}^{\frac{3}{2}}}
		\bigr)\,\dif t
		.
	\end{aligned}
\end{equation}

Using the same method as in Lemma~\ref{efa}, we have
\begin{equation}\label{rd1.1}
	\int_0^t
	\|(\tilde F_{\Psi},\tilde G_{\Psi}, \sqrt{\frac{\tau}{\mu}}\tilde H_{\Psi})\|_{\dot B_{2,1}^{\frac{3}{2}}}\,\dif t
\lesssim\int_0^t \|(a_{\Psi}^\tau,\mathbf{u}_{\Psi}^\tau,\mathbf{S}_{\Psi}^\tau)\|_{\dot{B}^{\frac{3}{2}}_{2,1}}
	\|(\tilde{a}_{\Psi},\tilde{\mathbf{u}}_{\Psi},\sqrt{\frac{\tau}{\mu}}\tilde{\mathbf{S}}_{\Psi})\|_{\dot{B}^{\frac{5}{2}}_{2,1}}\,\dif t.
\end{equation}

From Lemma~\ref{fge}, we get
\begin{equation}\label{rd1.2}
	\int_0^t
	\|\bar F_{\Psi}\|_{\dot B_{2,1}^{\frac{3}{2}}}\,\dif t
\lesssim\int_{0}^{t} \left(\|\tilde{\mathbf{u}}_{\Psi}\|_{\dot{B}^{\frac{3}{2}}_{2,1}}\|a^0_{\Psi}\|_{\dot{B}^{\frac{5}{2}}_{2,1}}
	+\|\tilde{a}_{\Psi}\|_{\dot{B}^{\frac{3}{2}}_{2,1}}\|\mathbf{u}^0_{\Psi}\|_{\dot{B}^{\frac{5}{2}}_{2,1}}\right)\,\dif t,
\end{equation}
\begin{equation}\label{rd1.3}
	\begin{aligned}
		\int_0^t
		\|\bar G_{\Phi}\|_{\dot B_{2,1}^{\frac{3}{2}}}\,\dif t
		&\lesssim\int_{0}^{t} \Bigl(\|\tilde{\mathbf{u}}_{\Psi}\|_{\dot{B}^{\frac{3}{2}}_{2,1}}\|\mathbf{u}^0_{\Psi}\|_{\dot{B}^{\frac{5}{2}}_{2,1}}
		+\|\tilde{a}_{\Psi}\|_{\dot{B}^{\frac{3}{2}}_{2,1}}\|a^0_{\Phi}\|_{\dot{B}^{\frac{5}{2}}_{2,1}}\\
		&\qquad+\|\tilde{a}_{\Psi}\|_{\dot{B}^{\frac{3}{2}}_{2,1}}\|\mathbf{u}^0_{\Psi}\|_{\dot{B}^{\frac{5}{2}}_{2,1}}\|a^0_{\Psi}\|_{\dot{B}^{\frac{5}{2}}_{2,1}}\Bigr)\,\dif s\\
		&\quad+ \int_{0}^{t} \|\tilde{\mathbf{S}}_{\Psi}\|_{\dot{B}^{\frac{3}{2}}_{2,1}}\|a^0_{\Psi}\|_{\dot{B}^{\frac{5}{2}}_{2,1}}\,\dif t,
	\end{aligned}
\end{equation}
and
\begin{equation}\label{rd1.4}
	\begin{aligned}
		\int_0^t
		\|\sqrt{\frac{\tau}{\mu}}\bar H_{\Psi}\|_{\dot B_{2,1}^{\frac{3}{2}}}\,\dif t
		&\lesssim  \sqrt{\tau}\int_{0}^{t}  \Bigl(\|\mathbf{u}^\tau_{\Psi}\|_{\dot{B}^{\frac{3}{2}}_{2,1}}\|\mathbf{u}^0_{\Psi}\|_{\dot{B}^{\frac{7}{2}}_{2,1}}
		+\|\mathbf{S}^\tau_{\Psi}\|_{\dot{B}^{\frac{3}{2}}_{2,1}}\|\mathbf{u}^0_{\Psi}\|_{\dot{B}^{\frac{5}{2}}_{2,1}}\Bigr)\,\dif t\\
		&\quad+ \sqrt{\tau}\int_{0}^{t} \Bigl(\|\mathbf{u}^0_{\Psi}\|_{\dot{B}^{\frac{9}{2}}_{2,1}}
		+\|a^0_{\Psi}\|_{\dot{B}^{\frac{7}{2}}_{2,1}}
		+\|\mathbf{u}^0_{\Psi}\|_{\dot{B}^{\frac{5}{2}}_{2,1}}^2
		+\|\mathbf{u}^0_{\Psi}\|_{\dot{B}^{\frac{3}{2}}_{2,1}}\|\mathbf{u}^0_{\Psi}\|_{\dot{B}^{\frac{7}{2}}_{2,1}}
		+\|a^0_{\Psi}\|_{\dot{B}^{\frac{5}{2}}_{2,1}}^2\Bigr)\,\dif t\\
		&\quad+ \sqrt{\tau}\int_{0}^{t} \Bigl(\|\mathbf{u}^0_{\Psi}\|_{\dot{B}^{\frac{5}{2}}_{2,1}}\|a^0_{\Psi}\|_{\dot{B}^{\frac{5}{2}}_{2,1}}
		+\|a^0_{\Psi}\|_{\dot{B}^{\frac{7}{2}}_{2,1}}\|\mathbf{u}^0_{\Psi}\|_{\dot{B}^{\frac{5}{2}}_{2,1}}
		+\|a^0_{\Psi}\|_{\dot{B}^{\frac{5}{2}}_{2,1}}\|\mathbf{u}^0_{\Psi}\|_{\dot{B}^{\frac{7}{2}}_{2,1}}\Bigr)\,\dif t.
	\end{aligned}
\end{equation}

Combining \eqref{rdxygj}, \eqref{rd1.1}–\eqref{rd1.4} and using \eqref{jdgj1}, choosing $\varepsilon_1$ sufficiently small, we obtain
\begin{align*}
	&\|(\tilde a_{\Psi},\tilde{\mathbf u}_{\Psi},\sqrt{\frac{\tau}{\mu}}\,\tilde{\mathbf S}_{\Psi})\|_{\dot B_{2,1}^{\frac{3}{2}}}
	+\nu\int_0^t\dot\eta(t)
	\|(\tilde a_{\Psi},\tilde{\mathbf u}_{\Psi},\sqrt{\frac{\tau}{\mu}}\,\tilde{\mathbf S}_{\Psi})\|_{\dot B_{2,1}^{\frac{5}{2}}}\,\dif t \\
	&\quad+\int_0^t\|(\tilde a_{\Psi},\tilde{\mathbf u}_{\Psi})\|_{\dot B_{2,1}^{\frac{5}{2}, \frac{3}{2}}}\,\dif t
	+\int_0^t\|\tilde{\mathbf S}_{\Psi}\|_{\dot B_{2,1}^{\frac{3}{2}}}\,\dif t\\
	&\lesssim
	\|(\tilde a_{\Psi},\tilde{\mathbf u}_{\Psi},\sqrt{\frac{\tau}{\mu}}\,\tilde{\mathbf S}_{\Psi})(0)\|_{\dot B_{2,1}^{\frac{3}{2}}}
	+\int_0^t \|(a_{\Psi}^\tau,\mathbf{u}_{\Psi}^\tau,\mathbf{S}_{\Psi}^\tau)\|_{\dot{B}^{\frac{3}{2}}_{2,1}}
	\|(\tilde{a}_{\Psi},\tilde{\mathbf{u}}_{\Psi},\sqrt{\frac{\tau}{\mu}}\tilde{\mathbf{S}}_{\Psi})\|_{\dot{B}^{\frac{5}{2}}_{2,1}}\,\dif t+\sqrt{\tau}.
\end{align*}
By choosing \(\nu\) sufficiently large, we can absorb the second term on the right-hand side into the left-hand side, and obtain
\begin{align*}
	&\|(\tilde a_{\Psi},\tilde{\mathbf u}_{\Psi},\sqrt{\frac{\tau}{\mu}}\,\tilde{\mathbf S}_{\Psi})\|_{\dot B_{2,1}^{\frac{3}{2}}}
	+\int_0^t\|(\tilde a_{\Psi},\tilde{\mathbf u}_{\Psi})\|_{\dot B_{2,1}^{\frac{5}{2}, \frac{3}{2}}}\,\dif t
	+\int_0^t\|\tilde{\mathbf S}_{\Psi}\|_{\dot B_{2,1}^{\frac{3}{2}}}\,\dif t \\
	&\lesssim
	\|(\tilde a_{\Psi},\tilde{\mathbf u}_{\Psi},\sqrt{\frac{\tau}{\mu}}\,\tilde{\mathbf S}_{\Psi})(0)\|_{\dot B_{2,1}^{\frac{3}{2}}}
	+ \sqrt{\tau}.
\end{align*}
By choosing $\varepsilon_0$ sufficiently small, we have $\nu \eta(t) \le \lambda_0$. Thus, we have $\lambda_0 |\xi|\le \Psi(t, \xi)\le 2\lambda_0 |\xi|$. This completes the proof of Theorem~\ref{th1.2}.

\textbf{Acknowledgement}: Yuxi Hu's research is supported by the Fundamental Research Funds for the Central Universities (No. 2026ZKPYLX01) .

\textbf{Disclosure of AI assistance:}  During the preparation of the manuscript, the authors used AI tools solely to assist with language polishing. The authors take full responsibility for the accuracy, originality, and integrity of the entire paper.

\end{document}